\documentclass[11pt]{article}
\usepackage[top=2cm, bottom=2cm, left=2.5cm, right=2.5cm] {geometry}
\usepackage{comment}
\usepackage{amsmath,amssymb,theorem,color}
\numberwithin{equation}{section} 
\usepackage{amssymb}
\usepackage{color}
\usepackage{amsmath}
\usepackage{graphicx}
\usepackage{amsfonts}%
{\Large }

\newcommand{\R}{\ensuremath{\mathbb{R}}}

\newcommand{\N}{\ensuremath{\mathbb{N}}}

\newcommand{\K}{\mathbb{K}}

\newcommand{\cC}{\mathcal{C}}
\newcommand{\cF}{\mathcal{F}}
\newcommand{\bP}{\mathbb{P}}

\newcommand{\mP}{\mathbb{P}}

\newcommand{\X}{\mathbb{X}}
\newcommand{\E}{\mathbb{E}}

\newcommand{\cD}{\mathcal{D}}
\newcommand{\cL}{\mathcal{L}}

\newcommand{\bx}{\mathbf x}
\newcommand{\bX}{\mathbf X}

\newcommand{\cB}{\mathcal{B}}
\newcommand{\ty}{\bar y}
\newcommand{\tp}{p{\rm -var}}
\newcommand{\tq}{q{\rm -var}}

{\Large }

\newcommand{\ltn}{\ensuremath{\left| \! \left| \! \left|}}
\newcommand{\rtn}{\ensuremath{\right| \! \right| \! \right|}}

\newtheorem{theorem}{Theorem}[section]
{ \theorembodyfont{\normalfont} 
	
	\newtheorem{remark}[theorem]{Remark}
}

\newtheorem{lemma}[theorem]{Lemma}
\newtheorem{corollary}[theorem]{Corollary}
\newtheorem{proposition}[theorem]{Proposition}

\newcounter{enumctr}

\begin{document}
\title{Asymptotic stability of controlled differential equations. \\Part II: rough integrals}

\author{Luu Hoang Duc \thanks{Max-Planck-Institute for Mathematics in the Sciences, Leipzig, Germany,
		\& Institute of Mathematics, Viet Nam Academy of Science and Technology	{\tt\small duc.luu@mis.mpg.de, lhduc@math.ac.vn}
	}
}
\date{}
\maketitle

\begin{abstract}
We continue the approach in Part I \cite{duchong19} to study stationary states of controlled differential equations driven by rough paths, using the framework of random dynamical systems and random attractors. Part II deals with driving paths of finite $\nu$ - H\"older norms with $\nu \in (\frac{1}{3},\frac{1}{2})$ so that the integrals are interpreted in the Gubinelli sense for controlled rough paths. We prove sufficient conditions for the attractor to be a singleton, thus the pathwise convergence is in both pullback and forward senses.
\end{abstract}



\section{Introduction}

This paper studies the asymptotic behavior of the rough differential equation
\begin{equation}\label{fSDE0}
dy_t = [Ay_t + f(y_t)]dt + g(y_t)d x_t,t\in \R_+,\ y(0)=y_0 \in \R^d, 
\end{equation}
where we assume for simplicity that $A \in \R^{d\times d}$, $f: \R^d \to \R^d, g: \R^d \to \R^{d\times m}$ are globally Lipschitz continuous functions with enough regularity, and  the driving path $x \in C^{\nu{\rm -Hol}}(\R,\R^m) \subset C^{p{\rm - var}}(\R, \R^m)$, with $\frac{1}{3}< \nu <\frac{1}{2}, p >\frac{1}{\nu}$ for simplicity. Such system is understood in the pathwise sense of a stochastic differential equation driven by a H\"older continuous stochastic process. 
In this circumstance, equation \eqref{fSDE0} is often solved using rough path theory, where the solution is understood in the sense of either Lyons \cite{lyons98}, \cite{lyonsetal07}, or of Friz-Victoir \cite{friz}, \cite{BRSch17}. However, since such definitions of rough differential equations need not to specify what a rough integral is, it is not clear how to apply the semigroup technique, which is well developed in \cite{ducGANSch18} and \cite{duchong19} for Young differential equations \cite{young}, to estimate rough integrals. In this paper, we would like to approach system \eqref{fSDE0} by considering the second integral as a rough integral for controlled rough paths in the sense of Gubinelli \cite{gubinelli}.  
System \eqref{fSDE0} can then be proved \cite{duc20} to admit a unique path-wise solution given initial conditions. 

Our aim is to investigate the role of the driving noise in the longterm behavior of rough system \eqref{fSDE0} as described in details in Part I \cite{duchong19} for Young integrals. Our main result in this paper is to prove that under strong dissipativity condition, system \eqref{fSDE0} also possesses a random attractor which would be a singleton provided that the Lipschitz coefficient $C_g$ in \eqref{gcond} is small enough. 
\subsubsection*{Assumptions and main results}
We impose the following assumptions for the coefficient functions and the driving path.\\

(${\textbf H}_1$) $A \in \R^{d\times d}$ is a matrix which has all eigenvalues of negative real parts;

(${\textbf H}_2$) $f: \R^d \to \R^d$ is globally Lipschitz continuous with the Lipschitz constant $C_f$. $g$ either belongs to $C^3_b$ such that 
\begin{equation}\label{gcond}
\|g\|_\infty, C_g := \max\Big\{\|Dg\|_\infty, \|D_g^2\|_\infty, \|D^3_g\|_\infty \Big\} < \infty,
\end{equation}
or has a linear form $g(y) = C y + g(0)$, where $C \in \R^d \otimes \R^{d\times m}$;

(${\textbf H}_3$) for a given $p \in (1,2)$, $x$ belongs to the space $\cC^{p{\rm - var}}(\R, \R^m)$ of all continuous paths which is of finite $p-$variation on any interval $[a,b]$. In particular, $x$ is a realization of a stationary stochastic process $Z_t(\omega)$, such that $x$ can be lifted into a realized component $\bx = (x,\X)$ of a stationary stochastic process $(x_\cdot(\omega),\X_{\cdot,\cdot}(\omega))$, such that the estimate
\[
E \Big(\|x_{s,t} \|^p +\|\X_{s,t}\|^{q}\Big)\leq C_{T,\nu} |t-s|^{p \nu },\forall s,t \in [0,T]
\]
holds for any $[0,T]$ for some constant $C_{T,\nu}$. As such $x \in C^{p{\rm - var}}(\R, \R^m)$, such that 
\begin{equation}\label{Gamma}
\Gamma(p):=\Big(E \ltn Z \rtn^p_{p{\rm -var},[-1,1]}\Big)^{\frac{1}{p}} < \infty. 
\end{equation}

Such a stochastic process, in particular, can be a fractional Brownian motion $B^H$  \cite{mandelbrot} with Hurst exponent $H \in (\frac{1}{3},1)$, i.e. a family of centered Gaussian processes $B^H = \{B^H_t\}_{t\in \R}$ with continuous sample paths and 
\[
E \|B^H_t- B^H_s\| = |t-s|^{2H}, \forall t,s \in \R.
\] 
Assumption (${\textbf H}_1$) ensures that there exist constant $C_A\geq 1,\lambda_A >0$ such that 
	\begin{eqnarray}
	\|\Phi\|_{\infty,[a,b]} &\leq& C_Ae^{-\lambda_A a}, \label{estphi1}\\
	\ltn \Phi\rtn_{p{\rm -var},[a,b]} &\leq& \|A\|C_A e^{-\lambda_A a}(b-a),\quad \forall\;  0\leq a<b, \label{estphi2}
	\end{eqnarray}
	where $\Phi(t) =e^{At}$ is the semigroup generated by $A$ \cite[Proposition 2.2]{duchong19}. Our main result can be formulated as follows.

\begin{theorem}\label{mainthm}
	Assume that the system \eqref{stochYDE} satisfies the assumptions ${\textbf H}_1-{\textbf H}_3$, and further that  $\lambda_A > C_fC_A$, where $\lambda_A$ and $C_A$ are given from \eqref{estphi1},\eqref{estphi2}. If
	\begin{equation}\label{criterion0}
	\lambda_A - C_A C_f> \frac{1}{2}C_A e^{\lambda_A+4(\|A\|+C_f)} \Big\{\Big[4C_pC_g\Gamma(p)\Big]^p + \Big[4C_pC_g\Gamma(p)\Big]\Big\},
	\end{equation}
	where $\Gamma(p)$ is defined in \eqref{Gamma} and $C_p$ in \eqref{roughpvar}, then the generated random dynamical system $\varphi$ of \eqref{stochYDE} possesses a pullback attractor $\mathcal{A}(x)$. Moreover, this attractor is a singleton, i.e. $\mathcal{A}(x) = \{a(x)\}$ a.s., in case $g(y) = Cy +g(0)$ is a linear map satisfying \eqref{criterion0} or in case $g \in C^2_b$ with the Lipschitz constant $C_g$ small enough, thus the pathwise convergence is in both the pullback and forward directions.	
\end{theorem}
The paper is organized as follows. Section 2 is devoted to present the existence, uniqueness and the norm estimates of the solution of rough system \eqref{fSDE0}. In subsection 3.1, we introduce the generation of random dynamical system by the equation \eqref{fSDE0}. Using Lemma \ref{ytest}, we prove the existence of a global random pullback attractor in Theorem \ref{attractor}. We also prove in Theorem \ref{linear} for case $g(y) = Cy$ and Theorem \ref{gbounded} for the case $C_g \in C^3_b$ that the attractor is a random singleton as long as $C_g \leq \delta$ is small enough. 

\section{Rough integrals}
Let us present in this preparation section a short summary on rough path theory and rough integrals. Given any compact time interval $I \subset \R$, let $C(I,\R^d)$ denote the space of all continuous paths $y:\;I \to \R^d$ equipped with sup norm $\|\cdot\|_{\infty,I}$ given by $\|y\|_{\infty,I}=\sup_{t\in I} \|y_t\|$, where $\|\cdot\|$ is the Euclidean norm in $\R^d$. We write $y_{s,t}:= y_t-y_s$. For $p\geq 1$, denote by $\cC^{p{\rm-var}}(I,\R^d)\subset C(I,\R^d)$ the space of all continuous path $y:I \to \R^d$ which is of finite $p$-variation 
\begin{eqnarray}
\ltn y\rtn_{p\text{-var},I} :=\left(\sup_{\Pi(I)}\sum_{i=1}^n \|y_{t_i,t_{i+1}}\|^p\right)^{1/p} < \infty,
\end{eqnarray}
where the supremum is taken over the whole class of finite partition of $I$. $\cC^{p{\rm-var}}(I,\R^d)$ equipped with the $p-$var norm
\begin{eqnarray*}
	\|y\|_{p\text{-var},I}&:=& \|y_{\min{I}}\|+\ltn y\rtn_{p\rm{-var},I},
\end{eqnarray*}
is a nonseparable Banach space \cite[Theorem 5.25, p.\ 92]{friz}. Also for each $0<\alpha<1$, we denote by $C^{\alpha}(I,\R^d)$ the space of H\"older continuous functions with exponent $\alpha$ on $I$ equipped with the norm
\[
\|y\|_{\alpha,I}: = \|y_{\min{I}}\| + \ltn y\rtn_{\alpha,I}=\|y(a)\| + \sup_{s<t\in I }\frac{\|y_{s,t}\|}{(t-s)^\alpha},
\]
A continuous map $\overline{\omega}: \Delta^2(I)\longrightarrow \R^+, \Delta^2(I):=\{(s,t): \min{I}\leq s\leq t\leq \max{I}\}$ is called a {\it control} if it is zero on the diagonal and superadditive, i.e.  $\overline{\omega}_{t,t}=0$ for all $t\in I$, and  $\overline{\omega}_{s,u}+\overline{\omega}_{u,t}\leq \overline{\omega}_{s,t}$ for all $s\leq u\leq t$ in $I$. 

We also introduce the construction of the integral using rough paths for the case $y,x \in C^\alpha(I)$ when $\alpha\in(\frac{1}{3},\nu)$. To do that, we need to introduce the concept of rough paths. Following \cite{frizhairer}, a couple $\bx=(x,\X)$, with $x \in C^\alpha(I,\R^m)$ and $\X \in C^{2\alpha}_2(\Delta^2(I),\R^m \otimes \R^m):= \{\X: \sup_{s<t} \frac{\|\X_{s,t}\|}{|t-s|^{2\alpha}} < \infty \}$ where the tensor product $\R^m \otimes \R^n$ can be indentified with the matrix space $\R^{m\times n}$, is called a {\it rough path} if it satisfies Chen's relation
\begin{equation}\label{chen}
\X_{s,t} - \X_{s,u} - \X_{u,t} = x_{s,u} \otimes x_{u,t},\qquad \forall \min{I} \leq s \leq u \leq t \leq \max{I}. 
\end{equation}
$\X$ is viewed as {\it postulating} the value of the quantity $\int_s^t x_{s,r} \otimes dx_r := \X_{s,t}$ where the right hand side is taken as a definition for the left hand side. Denote by $\cC^\alpha(I) \subset C^\alpha \oplus C^{2\alpha}_2$ the set of all rough paths in $I$, then $\cC^\alpha$ is a closed set but not a linear space, equipped with the rough path semi-norm 
\begin{equation}\label{translated}
\ltn \bx \rtn_{\alpha,I} := \ltn x \rtn_{\alpha,I} + \ltn \X \rtn_{2\alpha,\Delta^2(I)}^{\frac{1}{2}} < \infty.  
\end{equation}

Given fixed $\nu \in (\frac{1}{3},\frac{1}{2}), \alpha \in (\frac{1}{3},\nu)$ and  $p \in (\frac{1}{\alpha},3)$, on each compact interval $I$ such that $|I|=\max{I} - \min{I} \leq 1$, we also consider in this paper the rough path $\bx = (x,\X)$ with the $\tp$ norm 
\begin{equation}\label{pvarnorm}
\ltn \bx \rtn_{\tp,I} := \Big(\ltn x \rtn^p_{\tp,I} + \ltn \X \rtn_{\tq,I}^q\Big)^{\frac{1}{p}}, \quad \text{where\ }  q=\frac{p}{2}.  
\end{equation}

In the stochastic scenarios, it is often assumed \cite{frizhairer} that the driving path $x \in C^{\nu{\rm -Hol}}(I,\R^m)\subset C^{p{\rm - var}}(I, \R^m)$ can be lifted into a realized component $\bx = (x,\X)$ of a stationary stochastic process $(x_\cdot(\omega),\X_{\cdot,\cdot}(\omega))$, such that the estimate
\[
E \Big(\|x_{s,t} \|^p +\|\X_{s,t}\|^{q}\Big)\leq C_{T,\nu} |t-s|^{p \nu },\forall s,t \in [0,T]
\]
holds for any $[0,T]$ for some constant $C_{T,\nu}$. 

\subsection{Controlled rough paths}

Following \cite{gubinelli}, a path $y \in C^\alpha(I,\cL(\R^m,\R^d))$ is then called to be {\it controlled by} $x \in C^\alpha(I,\R^m)$ if there exists a tube $(y^\prime,R^y)$ with $y^\prime \in C^\alpha(I,\cL(\R^m,\cL(\R^m,\R^d))), R^y \in C^{2\alpha}(\Delta^2(I),\cL(\R^m,\R^d))$ such that
\begin{equation}\label{controlRP}
y_{s,t} = y^\prime_s \otimes x_{s,t} + R^y_{s,t},\qquad \forall \min{I}\leq s \leq t \leq \max{I}.
\end{equation}
$y^\prime$ is called Gubinelli derivative of $y$, which is uniquely defined as long as $x \in C^\alpha\setminus C^{2\alpha}$ (see \cite[Proposition 6.4]{frizhairer}). The space $\cD^{2\alpha}_x(I)$ of all the couple $(y,y^\prime)$ that is controlled by $x$ will be a Banach space equipped with the norm
\begin{eqnarray*}
	\|y,y^\prime\|_{x,2\alpha,I} &:=& \|y_{\min{I}}\| + \|y^\prime_{\min{I}}\| + \ltn y,y^\prime \rtn_{x,2\alpha,I},\qquad \text{where} \\
	\ltn y,y^\prime \rtn_{x,2\alpha,I} &:=& \ltn y^\prime \rtn_{\alpha,I} +   \ltn R^y\rtn_{2\alpha,I},
\end{eqnarray*}
where we omit the value space for simplicity of presentation. Now fix a rough path $(x,\X)$, then for any $(y,y^\prime) \in \cD^{2\alpha}_x (I)$, it can be proved that the function $F \in C^\alpha(\Delta^2 (I),\R^d)$ defined by 
\[
F_{s,t} := y_s \otimes x_{s,t} + y^\prime_s \otimes \X_{s,t}, \quad \forall s,t \in I, s \leq t,
\] 
satisfies
\[
F_{s,t} - F_{s,u}-F_{u,t} = - R^y_{s,u} \otimes x_{u,t} - y^\prime_{s,u} \otimes \X_{u,t}, \quad \forall s\leq u \leq t;
\]
hence it belongs to the space 
\begin{eqnarray*}
	C^{\alpha, 3\alpha}_2(I) &:=& \Big \{ F\in C^\alpha(\Delta^2(I)): F_{t,t} =0 \quad \text{and}\\ && \quad \qquad \qquad \qquad \qquad \ltn \delta F \rtn_{3\alpha,I} := \sup_{\min{I} \leq s \leq u \leq t \leq \max{I}} \frac{\|F_{s,t} - F_{s,u}-F_{u,t}\|}{|t-s|^{3\alpha}} < \infty \Big\}.
\end{eqnarray*}
Thanks to the sewing lemma (see e.g. \cite{gubinelli}, \cite[Lemma 4.2]{frizhairer}), the integral $\int_s^t y_u dx_u$ can be defined as  
\[
\int_s^t y_u dx_u := \lim \limits_{|\Pi| \to 0} \sum_{[u,v] \in \Pi} [ y_{u} \otimes x_{u,v} + y^\prime_u \otimes \X_{u,v} ]
\]
where the limit is taken on all the finite partition $\Pi$ of $I$ with $|\Pi| := \displaystyle\max_{[u,v]\in \Pi} |v-u|$ (see \cite{gubinelli}). Moreover, there exists a constant $C_\alpha = C_{\alpha,|I|} >1$ with $|I| := \max{I} - \min{I}$, such that
\begin{equation}\label{roughEst}
\Big\|\int_s^t y_u dx_u - y_s x_{s,t} + y^\prime_s \X_{s,t}\Big\| \leq C_\alpha |t-s|^{3\alpha} \Big(\ltn x \rtn_{\alpha,[s,t]} \ltn R^y \rtn_{2\alpha,\Delta^2[s,t]} + \ltn y^\prime\rtn_{\alpha,[s,t]} \ltn \X \rtn_{2\alpha,\Delta^2[s,t]}\Big).
\end{equation}
From now on, if no other emphasis, we will simply write $\ltn x \rtn_{\alpha}$ or $\ltn \X \rtn_{2\alpha}$ without addressing the domain in $I$ or $\Delta^2(I)$. As proved in \cite{gubinelli}, the rough integral of controlled rough paths follows the rule of integration by parts.
In practice, we would use the $p$-var norm
\begin{eqnarray*}
	\|(y,y^\prime)\|_{x,p,I} &:=& \|y_{\min{I}}\| + \|y^\prime_{\min{I}}\| + \ltn (y,y^\prime) \rtn_{x,p,I},\qquad \text{where} \\
	\ltn (y,y^\prime) \rtn_{x,p,I} &:=& \ltn y^\prime \rtn_{p{\rm -var},I} +   \ltn R^y\rtn_{\tq,I}.
\end{eqnarray*}
Thanks to the sewing lemma \cite{gubinelli}, we can use a similar version to \eqref{roughEst} under $p-$var norm as follows.
\begin{equation}\label{roughpvar}
\Big\|\int_s^t y_u dx_u - y_s x_{s,t} + y^\prime_s \X_{s,t}\Big\| \leq C_p \Big(\ltn x \rtn_{\tp,[s,t]} \ltn R^y \rtn_{\tq,\Delta^2[s,t]} + \ltn y^\prime\rtn_{\tp,[s,t]} \ltn \X \rtn_{\tq,\Delta^2[s,t]}\Big),
\end{equation}
with constant $C_p >1$ independent of $\bx$ and $y$. 

\subsection{Existence and uniqueness theorem}

In the following, we would like to construct a greedy sequence of stopping times as presented in \cite{cassetal}. 
Given $\frac{1}{p}\in (\frac{1}{3},\nu)$, we construct for any fixed $\gamma \in (0,1)$ the sequence of greedy times $\{\tau_i(\gamma,I,\tp)\}_{i \in \N}$ w.r.t. H\"older norms 
\begin{equation}\label{greedytime}
\tau_0 = \min{I},\quad \tau_{i+1}:= \inf\Big\{t>\tau_i:  \ltn \bx \rtn_{\tp, [\tau_i,t]} = \gamma \Big\}\wedge \max{I}.
\end{equation}
Denote by $N_{\gamma,I,p}(\bx):=\sup \{i \in \N: \tau_i \leq \max{I}\}$. It follows that
\begin{equation}\label{Nest}
N_{\gamma,I,p}(\bx) \leq 1 + \gamma^{-p} \ltn \bx \rtn^p_{\tp,I}.
\end{equation}
From now on, we would like to fix $\gamma = \frac{1}{4C_pC_g}$ and would like to write in short $N_{[a,b]}(\bx)$ for convenience. The existence and uniqueness theorem and norm estimates of the solution of \eqref{fSDE0} are proved in \cite[Theorem 3.1, Theorem 3.7, Theorem 3.8, Theorem 3.9]{duc20}.

\begin{theorem}\label{RDE}
	There exists a unique solution to \eqref{fSDE0} for any initial value, whose supremum and $p-$variation norms are estimated as follows
	\begin{eqnarray}
	\|y\|_{\infty,[a,b]} &\leq&  \Big[\|y_a\| + \Big( \frac{\|f(0)\|}{L}+ \frac{1}{C_p} \Big)N_{[a,b]}(\bx)\Big] e^{4L (b-a)}, \label{estx} \\
\|y_a\|+	\ltn y,R \rtn_{p{\rm -var},[a,b]} &\leq& \Big[\|y_a\| + \Big(\frac{\|f(0)\|}{L} + \frac{1}{C_p}\Big) N_{[a,b]}(\bx)\Big]e^{4L(b-a) }N^{\frac{p-1}{p}}_{[a,b]}(\bx) , \label{estx2}
	\end{eqnarray}
	where $L =\|A\|+C_f$ and $\ltn y,R \rtn_{\tp,[s,t]} :=\ltn y \rtn_{\tp,[s,t]} + \ltn R^y \rtn_{\tq,[s,t]}$. 
\end{theorem}

Following the same arguments line by line, we could prove similar estimates for $g = Cy$ as follows.
\begin{theorem}\label{RDElinear}
	There exists a unique solution to the rough differential equation 
	\begin{equation}\label{Existlinear1}
	dy_t = [Ay_t + f(y_t)]dt + \Big(C y_t + g(0)\Big) d x_t,t\in \R,\ y(0)=y_0 \in \R^d, 
	\end{equation}
	for any initial value, whose supremum and $p-$variation norms of the solution are estimated as follows
	\begin{eqnarray}\label{estxlin}
	\|y\|_{\infty,[a,b]} &\leq&  \Big[\|y_a\| + M_0 N_{[a,b]}(\bx)\Big] e^{4C_f (b-a) + L N_{[a,b]}(\bx)}, \notag \\
\|y_a\|+	\ltn y,R \rtn_{p{\rm -var},[a,b]} &\leq& \Big[\|y_a\| + M_0  N_{[a,b]}(\bx)\Big]e^{4C_f(b-a) + \alpha N_{[a,b]}(\bx)}N^{\frac{p-1}{p}}_{[a,b]}(\bx), 
	\end{eqnarray}
	where $\ltn y,R \rtn_{\tp,[s,t]} :=\ltn y \rtn_{\tp,[s,t]} + \ltn R^y \rtn_{\tq,[s,t]}$, $M_0=(1+\frac{3}{2C_p}) \frac{\|g(0)\|}{\|C\|}+\frac{\|f(0)\|}{C_f}$ and $\alpha = \log(1+ \frac{3}{2C_p})$. 
\end{theorem}

\begin{theorem}\label{RDEdifference}
	Consider two solutions $y_t(\bx,y_a)$ and $\ty_t(\bx,\ty_a)$ and their difference $z_t = \ty_t -y_t$, which satisfies the integral rough equation
	\begin{equation}\label{RDEdiff}
	z_t = z_a +\int_a^t [f(\ty_s)-f(y_s)]ds+\int_a^t [g(\ty_s) - g(y_s)]dx_s.
	\end{equation}
	Then
	\begin{equation}\label{roughest4}
\|z_a\| +	\ltn z,R \rtn_{p{\rm -var},[a,b]} \leq \|z_a\|  e^{4L(b-a)} \Big(1+ \Big[8C_p C_g \Lambda(\bx,[a,b])\Big]^{p-1} \ltn \bx \rtn_{p{\rm -var},[a,b]}^{p-1} \Big),
	\end{equation}
	where 
	\begin{equation}\label{Lambda}
	\Lambda(\bx,[a,b]) =1 +2 \Big[\|y_a\| \vee \|\ty_a\|+ \Big(\frac{\|f(0)\|}{C_f} + \frac{1}{C_p}\Big) N_{[a,b]}(\bx)\Big]e^{4L(b-a) }N^{\frac{p-1}{p}}_{[a,b]}(\bx).
	\end{equation}
\end{theorem}


\section{Random attractors}

\subsection{Generation of rough cocycle and rough flows}

In this subsection we would like to present the generation of a random dynamical system from rough differential equations \eqref{fSDE0}, which is based mainly on the work in \cite{BRSch17} with only a small modification. Recall that a {\it random dynamical system} is defined by a mapping $\varphi(t,\omega)x_0 := x(t,\omega,x_0)$ on the probability space $(\Omega,\mathcal{F},\mP)$ equipped with a metric dynamical system $\theta$, i.e. $\theta_{t+s} = \theta_t \circ \theta_s$ for all $t,s \in \R$, such that $\varphi: \R \times \Omega \times \R^d \to \R^d$ is a measurable mapping which is also continuous in $t$ and $x_0$ and satisfies the cocycle property
\[
\varphi(t+s,\omega) =\varphi(t,\theta_s \omega) \circ \varphi(s,\omega),\quad \forall t,s \in \R, 
\] 
(see \cite{arnold}). Now denote by $T^2_1(\R^m) = 1 \oplus \R^m \oplus (\R^m \otimes \R^m)$ the set with the tensor product
\[
(1,g^1,g^2) \otimes (1,h^1,h^2) = (1, g^1 + h^1, g^1 \otimes h^1 + g^2 +h^2),\quad \forall\  {\bf g} =(1,g^1,g^2), {\bf h} = (1,h^1,h^2) \in T^2_1(\R^m).
\]
Then it can be shown that $(T^2_1(\R^m),\otimes)$ is a topological group with unit element ${\bf 1} = (1,0,0)$. \\
For $\beta \in (\frac{1}{p},\nu)$, denote by $\cC^{0,p-\rm{var}}([a,b],T^2_1(\R^m))$ 
the closure of $\cC^{\infty}([a,b],T^2_1(\R^m))$ in $\cC^{\tp}([a,b],T^2_1(\R^m))$, and by 
$\cC_0^{0,p-\rm{var}}(\R,T^2_1(\R^m))$ the space of all $x: \R\to \R^m$ such that $x|_I \in \cC^{0,\tp}(I, T^2_1(\R^m))$ for each compact interval $I\subset\R$ containing $0$. Then $\cC_0^{0,p-\rm{var}}(\R,T^2_1(\R^m))$ is equipped with the compact open topology given by the $p-$variation norm, i.e  the topology generated by the metric:
\[
d_p(\bx_1,\bx_2): = \sum_{k\geq 1} \frac{1}{2^k} (\|\bx_1-\bx_2\|_{p{\rm-var},[-k,k]}\wedge 1),
\]
where the $\tp$ norm is given in \eqref{pvarnorm}. As a result, it is separable and thus a Polish space. \\
Let us consider a stochastic process $\bar{\bX}$ defined on a probability space $(\bar{\Omega},\bar{\mathcal{F}},\bar{\bP})$  with realizations in $(\cC^{0,\tp}_0(\R,T^2_1(\R^m)), \mathcal{F})$. Assume further that $\bar{\bX}$ has stationary increments. Assign
\[
\Omega:=\cC_0^{0,\tp}(\R,T^2_1(\R^m))
\]
and equip with the Borel $\sigma -$ algebra $\mathcal{F}$ and let $\bP$ be the law of $\bar{\bX}$. Denote by $\theta$ the {\it Wiener-type shift}
\begin{equation}\label{shift}
(\theta_t \omega)_\cdot = \omega_t^{-1}\otimes \omega_{t+\cdot},\forall t\in \R, \omega \in \cC^{0,\tp}_0(\R,T^2_1(\R^m)),
\end{equation}  
and define the so-called {\it diagonal process}  $\bX: \R \times \Omega \to T^2_1(\R^m), \bX_t(\omega) = \omega_t$ for all $t\in \R, \omega \in \Omega$. Due to the stationarity of $\bar{\bX}$, it can be proved that $\theta$ is invariant under $\bP$, then forming a continuous (and thus measurable) dynamical system on $(\Omega, \mathcal{F},\bP)$ \cite[Theorem 5]{BRSch17}. Moreover, $\bX$ forms a {\it $p-$ rough path cocycle}, namely, $\bX_\cdot(\omega) \in \cC_0^{0,\tp}(\R,T^2_1(\R^m))$ for every $\omega \in \Omega$, which satisfies the {\it cocyle relation}:
\[
\bX_{t+s}(\omega) = \bX_s(\omega) \otimes \bX_t(\theta_s \omega), \forall  \omega \in \Omega, t,s\in \R,
\]
in the sense that $\bX_{s,s+t} = \bX_t(\theta_s \omega)$ with the increment notation $\bX_{s,s+t} := \bX^{-1}_s \otimes \bX_{s+t}$. It is important to note that the two-parameter flow property
\[
\bX_{s,u} \otimes \bX_{u,t} = \bX_{s,t}, \forall s,t \in \R 
\]
is equivalent to the fact that $\bX_t(\omega) = (1,x_t(\omega),\X_{0,t}(\omega))$, where  $x_\cdot(\omega): \R \to \R^m$ and $\X_{\cdot,\cdot}(\omega): I \times I \to \R^m \otimes \R^m$ are random funtions satisfying Chen's relation relation \eqref{chen}. To fulfill the H\"older continuity of almost all realizations, assume further that for any given $T>0$, there exists a constant $C_{T,\nu}$ such that
\begin{equation}\label{expect}
E \Big(\|x_{s,t} \|^p +\|\X_{s,t}\|^{q}\Big)\leq C_{T,\nu} |t-s|^{p \nu },\forall s,t \in [0,T].
\end{equation}
Then due to the Kolmogorov criterion for rough paths \cite[Appendix A.3]{friz}, for any $\beta\in (\frac{1}{p},\nu)$ there exists a version of $\omega-$wise $(x,\X)$ and random variables $K_\beta \in L^p, \K_\beta \in L^{\frac{p}{2}}$, such that, $\omega-$wise speaking, for all $s,t \in I$,
\[
\|x_{s,t}\| \leq K_\alpha |t-s|^\beta, \quad \|\X_{s,t}\| \leq \K_\beta |t-s|^{2\beta}, \forall s,t \in \R
\] 
so that $(x,\X) \in \cC^\beta.$ Moreover, we could choose $\beta$ such that 
\begin{eqnarray*}
	&&x \in C^{0,\beta}(I):= \{x \in C^\beta: \lim \limits_{\delta \to 0}\sup_{0<t-s <\delta} \frac{\|x_{s,t}\|}{|t-s|^\beta} = 0\}, \\
	&&\X \in C^{0,2\beta}(\Delta^2(I)):= \{ \X \in C^{2\beta}(\Delta^2(I)): \lim \limits_{\delta \to 0}\sup_{0<t-s <\delta} \frac{\|\X_{s,t}\|}{|t-s|^{2\beta}} = 0  \}, 
\end{eqnarray*}
then $\cC^{0,\beta}(I) \subset C^{0,\beta}(I) \oplus C^{0,2\beta}(\Delta^2(I))$ is separable due to the separability of $C^{0,\beta}(I)$ and $C^{0,2\beta}(\Delta^2(I))$.  In particular, due to the fact that $\ltn \bX_\cdot(\theta_h \omega) \rtn_{\tp,[s,t]} = \ltn \bX_\cdot(\omega) \rtn_{\tp,[s+h,t+h]}$, it follows from  Birkhorff ergodic theorem and \eqref{expect} that
\begin{equation}\label{gamma}
\Gamma(\bx,p) := \limsup \limits_{n \to \infty} \Big(\frac{1}{n}\sum_{k=1}^{n}  \ltn \theta_{-k}\bx \rtn^p_{p{\rm -var},[-1,1]}\Big)^{\frac{1}{p}} =\Big(E \ltn \bX_\cdot(\cdot) \rtn^p_{\tp,[-1,1]}\Big)^{\frac{1}{p}}= \Gamma(p)
\end{equation}	
for almost all realizations $\bx_t$ of the form $\bX_t(\omega)$. We assume additionally that  $(\Omega, \mathcal{F}, \bP,\theta) $  is ergodic. 

\begin{remark}
 It is important to note that, due to \cite[Corollary 9]{BRSch17}, this construction is possible for $X: \R \to \R^m$ to be a continuous, centered Gaussian process with stationary increments and independent components, satisfying: there exists for any $T>0$ a constant $C_T$ such that for all $p \geq \frac{1}{\bar{\nu}}$
\begin{equation}\label{Gaussianexpect}
E \|X_t- X_s\|^{p} \leq C_T |t-s|^{p\nu},\quad \forall s,t \in [0,T]. 
\end{equation}
By Kolmogorov theorem, for any $\beta \in (\frac{1}{p},\nu)$ and any interval $[0,T]$ almost all realization of $X$ will be in $C^{0,\beta}([0,T])$. Then $X$ has its covariance function with finite 2-dimensional $\rho-$variation on every square $[s,t]^2 \in \R^2$ for some $\rho \in [1,2)]$, and $\bar{\bX}$ is the natural lift of $X$, in the sense of Friz-Victoir \cite[Chapter 15]{friz}, with sample paths in the space $\cC_0^{0,\beta-\rm{Hol}}(\R,T^2_1(\R^m))$, for every $p > 2\rho$.\\ For instance, such a stochastic process $X$, in particular, can be a $m-$ dimensional fractional Brownian motion $B^H$ with independent components \cite{mandelbrot} and Hurst exponent $H \in (\frac{1}{3},\frac{1}{2})$, i.e. a family of $B^H = \{B^H_t\}_{t\in \R}$ with continuous sample paths and 
\[
E [B^H_t B^H_s] = \frac{1}{2}\Big(t^{2H} + s^{2H} - |t-s|^{2H}\Big) I^{m \times m}, \forall t,s \in \R_+.
\]    
For any fixed interval $[0,T]$, the covariance of increments of fractional Brownian motions $R: [0,T]^4 \to \R^{m \times m}$, defined by 
\[
R \Big(\begin{array}{cc} s&t \\ s^\prime & t^\prime \end{array}\Big)  := E(B^H_{s,t} B^H_{s^\prime,t^\prime})
\]
is of finite $\rho-$ variation norm for $\rho = \frac{1}{2H}$, i.e.
\[
\|R\|_{I \times I^\prime,\rho} := \Big\{\sup_{\Pi(I), \Pi^\prime(I^\prime)} \sum_{[s,t] \in I, [s^\prime, t^\prime] \in I^\prime} \Big|R \Big(\begin{array}{cc} s&t \\ s^\prime & t^\prime \end{array}\Big)  \Big|^\rho \Big\}^{\frac{1}{\rho}} < \infty,
\]
and 
\[
\|R\|_{[s,t]^2,\rho} \leq M_{\rho,T} |t-s|^{\frac{1}{\rho}}, \forall t,s \in [0,T].
\]
Then one can prove that the integral in $L^2-$ sense 
\[
\X_{s,t}^{i,j} = \lim \limits_{|\Pi| \to 0} \int_\Pi X^i_{s,r} dX^j_r = \lim \limits_{|\Pi| \to 0}  \sum_{[u,v] \in \Pi} X^i_{s,u} X^j_{u,v}, \forall s,t \in [0,T] 
\]
is well-defined regardless of the chosen partition $\Pi$ of $[s,t]$. Moreover, 
\[
\X_{s,t}^{i,i} = \frac{1}{2} (X^i_{s,t})^2,\quad \X_{s,t}^{i,j} + \X_{s,t}^{j,i} = X^i_{s,t} X^j_{s,t},  
\]
and for $\frac{1}{p}<\nu < \frac{1}{2\rho} = H$, there exist constants $C(p,\rho,m,T), C(p,\rho,m,T,\nu)>0$ such that
\begin{eqnarray}
E\Big[\|X_{s,t}\|^p + \|\X_{s,t}\|^q \Big] &\leq& C(p,\rho,m,T) |t-s|^{\frac{p}{2\rho}} = C(p,\rho,m,T) |t-s|^{pH},\quad \forall s,t \in [0,T], \notag\\
E \Big[\ltn X\rtn_{\nu}^p + \ltn \X \rtn_{2\nu}^q \Big] &\leq& C(p,\rho,m,T,\nu) M^q.
\end{eqnarray}
Therefore, almost sure all realizations $\bx = (X,\X)$ belong to the set $\cC^\beta([0,T])$ and satisfy Chen's relation \eqref{chen}.
\end{remark}

We reformulate a result from \cite[Theorem 21]{BRSch17} for our situation as follows.
\begin{proposition}
	Let $(\Omega,\cF,\bP, \theta)$ be a measurable metric dynamical system and let $\bX: \R \times \Omega \to T^2_1(\R^m)$ be a $p$- rough cocycle for some $2\leq p <3$. Then there exists a unique continuous random dynamical system $\varphi$ over $(\Omega,\cF,\bP, \theta)$ which solves the rough differential equation 
	\begin{equation}\label{stochYDE}
	dy_t = [Ay_t + f(y_t)]dt + g(y_t)d \bX_t(\omega), t \geq 0.
	\end{equation}
\end{proposition}


\subsection{Existence of pullback attractors}

Given a random dynamical system $\varphi$ on $\R^d$, we follow \cite{crauelkloeden}, \cite[Chapter 9]{arnold} to present the notion of random pullback attractor. Recall that a set $\hat{M} :=
\{M(x)\}_{x \in \Omega}$ a {\it random set}, if $y
\mapsto d(y|M(x))$ is $\cF$-measurable for each $y \in \R^d$, where $d(E|F) = \sup\{\inf\{d(y, z)|z \in F\} | y \in E\}$  for $E,F$ are nonempty subset of $\R^d$ and $d(y|E) = d(\{y\}|E)$.  
An {\it universe} $\cD$ is a family of random sets
which is closed w.r.t. inclusions (i.e. if $\hat{D}_1 \in \cD$ and
$\hat{D}_2 \subset \hat{D}_1$ then $\hat{D}_2 \in \cD$). \\
In our setting, we define the universe $\cD$ to be a family of {\it tempered} random sets $D(x)$, which means the following: A random variable $\rho(x) >0$ is called {\it tempered} if it satisfies
\begin{equation}\label{tempered}
\lim \limits_{t \to \pm \infty} \frac{1}{t} \log^+ \rho(\theta_{t}x) =0,\quad \text{a.s.}
\end{equation}
(see e.g. \cite[pp. 164, 386]{arnold}) which, by \cite[p. 220]{ImkSchm01}), is equivalent to the sub-exponential growth
\[
\lim \limits_{t \to \pm \infty} e^{-c |t|} \rho(\theta_{t}x) =0\quad \text{a.s.}\quad \forall c >0.
\]
A random set $D(x)$ is called {\it tempered} if it is contained in a ball $B(0,\rho(x))$ a.s., where the radius $\rho(x)$ is a tempered random variable.\\
A random subset $A$ is called invariant, if $\varphi(t,x)A(x) = A(\theta_t x)$ for all $t\in \R,\; x\in\Omega.$ An invariant random compact set $\mathcal{A}  \in \cD$ is called a {\it pullback random attractor} in $\cD$, if $\mathcal{A} $ attracts
any closed random set $\hat{D} \in \cD$ in the pullback sense,
i.e.
\begin{equation}\label{pullback}
\lim \limits_{t \to \infty} d(\varphi(t,\theta_{-t}x)
\hat{D}(\theta_{-t}x)| \mathcal{A} (x)) = 0.
\end{equation}
$\mathcal{A} $ is called a {\it forward random attractor} in $\cD$, if $\mathcal{A} $ is invariant and attracts
any closed random set $\hat{D} \in \cD$ in the forward sense,
i.e.
\begin{equation}\label{forward}
\lim \limits_{t \to \infty} d(\varphi(t,x)
\hat{D}(x)| \mathcal{A} (\theta_{t}x)) = 0.
\end{equation}
The existence of a random pullback attractor follows from the existence of a random pullback absorbing set (see \cite[Theorem 3]{crauelkloeden}). A random set $\mathcal{B}  \in \cD$ is called {\it pullback
	absorbing} in a universe $\cD$ if $\mathcal{B} $ absorbs all sets in
$\cD$, i.e. for any $\hat{D} \in \cD$, there exists a time $t_0 =
t_0(x,\hat{D})$ such that
\begin{equation}\label{absorb}
\varphi(t,\theta_{-t}x)\hat{D}(\theta_{-t}x) \subset
\mathcal{B} (x), \ \textup{for all}\  t\geq t_0.
\end{equation}
Given a universe $\cD$ and a random compact
pullback absorbing set $\mathcal{B} \in \cD$, there exists a unique random pullback attractor
in $\cD$, given by
\begin{equation}\label{at}
\mathcal{A}(x) = \cap_{s \geq 0} \overline{\cup_{t\geq s} \varphi(t,\theta_{-t}x)\mathcal{B}(\theta_{-t}x)}. 
\end{equation}
Thanks to the rule of integration by parts for rough integral, we prove the "variation of constants" formula for rough differential equations as below.
\begin{lemma}
	The solution $y_t$ of \eqref{fSDE0} satisfies
	\begin{equation}\label{variation}
	y_t =\Phi (t-a)y_a + \int_a^t\Phi (t-s)f(y_s) ds + \int_a^t\Phi (t-s)g(y_s) d x_s,\quad \forall t\geq a.
	\end{equation} 
\end{lemma}
\begin{proof}
	Assign $z_u := \Phi(-u) y_u$. Observe that $\frac{d}{du}\Phi(u) = A \Phi(u)$ and $A\Phi(u) = \Phi(u)A$ for all $u\in \R$. As a result, we can write in the discrete form using \eqref{roughpvar} 
	\begin{eqnarray}\label{variationeq1}
	&& \Phi(-v)y_v - \Phi(-u)y_u \notag\\ 
	&=& [\Phi(-v) -\Phi(-u)]y_u + \Phi(-u) y_{u,v} +[\Phi(-v) -\Phi(-u)]y_{u,v} \notag\\
	&=& - A \Phi(-u)y_u (v-u) + \Phi(-u) \Big[A y_u + f(y_u)\Big] (v-u) \notag\\
	&& + \Phi(-u) \Big[g(y_u) \otimes x_{u,v} + Dg(y_u)g(y_u) \otimes \X_{u,v} \Big] + \mathcal{O}(|v-u|^{3\alpha}) \notag\\
	&=& \Phi(-u) f(y_u)(v-u)+ \Phi(-u) \Big[g(y_u) \otimes x_{u,v} + [g(y)]^\prime_u \otimes \X_{u,v} \Big] + \mathcal{O}(|v-u|^{3\alpha}). 
	\end{eqnarray}
	On the other hand,
	\[
	\Phi(-v)g(y_v) - \Phi(-u)g(y_u) = [\Phi(-v) -\Phi(-u)]g(y_v) + \Phi(-u) \Big( [g(y)]^\prime_u \otimes x_{u,v} + R^{g(y)}_{u,v}\Big),
	\]
	thus $\Phi(-\cdot)g(y)$ is also controlled by $x$ with $[\Phi(-\cdot)g(y_\cdot)]^\prime_u = \Phi(-u)[g(y)]^\prime_u $. Thus we can rewrite \eqref{variationeq1} as
	\begin{equation}\label{variationeq2}
	\Phi(-v)y_v - \Phi(-u)y_u =  \Phi(-u) f(y_u)(v-u)+ \Big([\Phi(-u)g(y_u)] \otimes x_{u,v} + [\Phi(-\cdot)g(y)]^\prime_u \otimes \X_{u,v} \Big) + \mathcal{O}(|v-u|^{3\alpha}). 
	\end{equation}
	Next, for any fixed $a,t \in R_+$, consider any finite partition $\Pi$ of $[a,t]$ such that $|\Pi| = \max_{[u,v]\in \Pi} |v-u| \ll 1$. It follows from \eqref{variationeq2} that
	\begin{eqnarray}\label{variationeq3}
	\Phi(-t)y_t - \Phi(-a)y_a &=& \sum_{[u,v]\in \Pi} \Phi(-u) f(y_u)(v-u) \\
	&& + \sum_{[u,v]\in \Pi} \Big([\Phi(-u)g(y_u)] \otimes x_{u,v} + [\Phi(-\cdot)g(y)]^\prime_u \otimes \X_{u,v} \Big) + \mathcal{O}(|\Pi|^{3\alpha-1}). \notag
	\end{eqnarray}
Let $|\Pi|$ tend to zero, the first sum in the right hand side of \eqref{variationeq3} converges to $\int_a^t \Phi(-s)f(y_s)ds$, the second sum converges to the rough integral $\int_a^t\Phi (-s)g(y_s) d x_s$, while the residual term $\mathcal{O}(|\Pi|^{3\alpha-1})$ converges to zero. Finally, by multiplying both sides by $\Phi(t)$ and using the semigroup propertiy, we obtain \eqref{variation}.

\end{proof}	
By the same arguments, one can show that.
\begin{corollary}\label{RDEvariationdiff}
		Consider two solutions $y_t(\bx,y_a)$ and $\ty_t(\bx,\ty_a)$ and their difference $z_t = \ty_t -y_t$ satisfies \eqref{RDEdiff}, as presented in Theorem \ref{RDEdifference}. Then $z_t$ also satisfies the variation of constants formula
			\begin{equation}\label{variationDiff}
		z_t =\Phi (t-a)z_a + \int_a^t\Phi (t-s)[f(y_s+z_s)-f(y_s)] ds + \int_a^t\Phi (t-s)[g(y_s+z_s)-g(y_s)]d x_s,\quad \forall t\geq a.
		\end{equation} 
		
\end{corollary}

We need the following auxiliary results. The first one come from  \cite[Proposition 3.2]{duchong19}.
\begin{proposition}\label{YDEg}
	Given \eqref{estphi1} and \eqref{estphi2}, the following estimate holds: for any $0\leq a<b\leq c$
	\begin{equation}\label{int2}
	 \Big\|\int_a^{b}\Phi (c-s)g( y_s) d  x_s\Big\|
	 \leq e^{-\lambda_A(c-b)}  \kappa(\bx,[a,b]) \Big( \frac{\|g(0)\|}{C_g} +\|y_a\| +\ltn y,R \rtn_{\tp,[a,b]}\Big) ,
	\end{equation}
where 
\begin{eqnarray}
\kappa(\bx,[a,b]):= 2C_p C_A[1+\|A\|(b-a)] \Big\{C_g^2 \ltn \bx \rtn^2_{\tp,[a,b]} \vee C_g \ltn \bx \rtn_{\tp,[a,b]}\Big\}. \label{kappa2}
\end{eqnarray}	
\end{proposition}

\begin{proof}
{\bf Case 1}. We first prove \eqref{int2} for the case $g \in C^3_b$. Observe that $y$ is controlled by $x$ with $y^\prime = g(y)$. Since
\begin{eqnarray*}
	&& g(y_t) - g(y_s) \\
	&=& \int_0^1 Dg(y_s + \eta y_{s,t}) y_{s,t}d\eta \\
	&=& D_g(y_s) y^\prime_s \otimes  x_{s,t} + \int_0^1 Dg(y_s + \eta y_{s,t}) R^y_{s,t}d\eta  + \int_0^1 [Dg(y_s + \eta y_{s,t}) - Dg(y_s)] y^\prime_{s,t} \otimes x_{s,t}d\eta,
\end{eqnarray*}
it easy to show that $[g(y)]^\prime_s = Dg(y_s)g(y_s)$, where we use \eqref{gcond} to estimate
\begin{eqnarray*}
	\|R^{g(y)}_{s,t}\| &\leq& \int_0^1 \|Dg(y_s + \eta y_{s,t})\| \|R^y_{s,t}\|d\eta + \int_0^1 \|Dg(y_s + \eta y_{s,t}) - Dg(y_s)\| \|g(y_s)\| \|x_{s,t}\|d\eta \\
	&\leq& C_g \|R^y_{s,t}\| + \frac{1}{2} C_g^2 \|y_{s,t}\| \|x_{s,t}\|.
\end{eqnarray*}
This together with H\"older inequality yields
\allowdisplaybreaks
\begin{eqnarray}\label{Exist2}
\ltn [g(y)]^\prime \rtn_{\tp,[s,t]} &\leq& 2 C_g^2 \ltn y \rtn_{\tp,[s,t]}, \quad \|[g(y)]^\prime\|_{\infty,[s,t]} \leq C_g^2, \notag\\
\ltn R^{g(y)} \rtn_{\tq,[s,t]} &\leq& C_g \ltn R^y \rtn_{\tq,[s,t]} + \frac{1}{2}C_g^2 \ltn x \rtn_{\tp,[s,t]} \ltn y \rtn_{\tp,[s,t]}.
\end{eqnarray}
Now because
	\[
\Phi(c-t)g(y_t) - \Phi(c-s)g(y_s) = [\Phi(c-t) -\Phi(c-s)]g(y_t) + \Phi(c-s) \Big( [g(y)]^\prime_s \otimes x_{s,t} + R^{g(y)}_{s,t}\Big),
\]
it follows that
\begin{eqnarray*}
	[\Phi(c-\cdot)g(y_\cdot)]^\prime_s &=& \Phi(c-s)[g(y)]^\prime_s = \Phi(c-s) Dg(y_s)g(y_s), \\
	\Big\|R^{\Phi(c-\cdot)g(y_\cdot)}_{s,t}\Big\| &\leq& \|\Phi(c-s)R^{g(y)}_{s,t}\| + \|\Phi(c-t) - \Phi(c-s)\| \|g(y_t)\|.
\end{eqnarray*}	
A direct computation shows that
\begin{eqnarray*}
\ltn [\Phi(c-\cdot)g(y_\cdot)]^\prime \rtn_{\tp,[a,b]} &\leq& 2 C_g^2 \|\Phi(c-\cdot)\|_{\infty,[a,b]} \ltn y \rtn_{\tp,[a,b]} + C_g^2 \ltn \Phi(c-\cdot) \rtn_{\tp,[a,b]}\\
\ltn R^{\Phi(c-\cdot)g(y_\cdot)} \rtn_{\tq,[a,b]} &\leq& \|\Phi(c-b)\| \ltn R^{g(y)} \rtn_{\tq,[a,b]} + C_g \ltn \Phi(c-\cdot) \rtn_{\tq,[a,b]}.
\end{eqnarray*}
Using \eqref{roughpvar} and \eqref{estphi1}, \eqref{estphi2}, we can now estimate
\allowdisplaybreaks
\begin{eqnarray*}
&&\Big\|\int_a^b\Phi (c-s)g(y_s) d  x_s\Big\|\notag\\
&\leq & \|\Phi(c-a)g(y_a)\| \|x_{a,b}\| + \|[\Phi(c-\cdot)g(y_\cdot)]^\prime_a \| \|\X_{a,b}\| \\
&&+ C_p \Big\{\ltn x \rtn_{\tp,[a,b]} \ltn R^{\Phi(c-\cdot)g(y_\cdot)} \rtn_{\tq,[a,b]} + \ltn \X \rtn_{\tq,[a,b]} \ltn [\Phi(c-\cdot)g(y_\cdot)]^\prime \rtn_{\tp,[a,b]}\Big\}\\
&\leq& C_A C_g e^{-\lambda_A(c-a)}\ltn x\rtn_{\tp,[a,b]} + C_A C_g^2 e^{-\lambda_A(c-a)}\ltn \X\rtn_{\tq,[a,b]}\\
&&+ C_p \ltn \X \rtn_{\tq,[a,b]} \Big[2C_AC_g^2 e^{-\lambda_A(c-b)}\ltn y \rtn_{\tp,[a,b]} + C_A C_g^2 \|A\|e^{-\lambda_A(c-b)}(b-a) \Big]\\
&&+ C_p \ltn x\rtn_{\tp,[a,b]} \Big\{ C_AC_g \|A\|e^{-\lambda_A(c-b)}(b-a) \\
&&+ C_A e^{-\lambda_A(c-b)} \Big( C_g \ltn R^y \rtn_{\tq,[a,b]} + \frac{1}{2}C_g^2 \ltn x \rtn_{\tp,[a,b]} \ltn y \rtn_{\tp,[a,b]} \Big) \Big\}\\
&\leq& C_A [1+C_p\|A\|(b-a)] e^{-\lambda_A(c-b)} \Big(C_g \ltn x \rtn_{\tp,[a,b]} + C_g^2 \ltn \X \rtn_{\tq,[a,b]} \Big)\\
&&+ C_p C_A e^{-\lambda_A(c-b)} \Big\{\Big[2C_g^2 \ltn \X \rtn_{\tq,[a,b]} + \frac{1}{2}C_g^2 \ltn x \rtn_{\tp,[a,b]}^2\Big] \vee C_g \ltn x \rtn_{\tp,[a,b]}\Big\} \ltn y,R \rtn_{\tp,[a,b]}.
\end{eqnarray*}
which, together proves \eqref{int2}.\\

{\bf Case 2}. Next, consider the case $g = Cy + g(0)$ with $\|C\| \leq C_g$. Then similar estimates show that 
\begin{eqnarray*}
\|R^{g(y)}_{s,t}\| &\leq& \|C\| \|R^y_{s,t}\|;\\
\ltn [\Phi(c-\cdot)g(y_\cdot)]^\prime \rtn_{\tp,[a,b]} &\leq&  \|C\|^2 \|\Phi(c-\cdot)\|_{\infty,[a,b]} \ltn y \rtn_{\tp,[a,b]}\\
&& + \ltn \Phi(c-\cdot) \rtn_{\tp,[a,b]}  \|C\| (\|C\| \|y\|_{\tp,[a,b]} + \|g(0)\|);\\
\ltn R^{\Phi(c-\cdot)g(y_\cdot)} \rtn_{\tq,[a,b]} &\leq& \|\Phi(c-b)\| \ltn R^{g(y)} \rtn_{\tq,[a,b]} \\
&&+  \ltn \Phi(c-\cdot) \rtn_{\tq,[a,b]} (\|C\| \|y\|_{\tp,[a,b]} + \|g(0)\|).
\end{eqnarray*}
As a result,
\begin{eqnarray*}
	&&\Big\|\int_a^b\Phi (c-s)g(y_s) d  x_s\Big\|\notag\\
	&\leq & \|\Phi(c-a)g(y_a)\| \|x_{a,b}\| + \|[\Phi(c-\cdot)g(y_\cdot)]^\prime_a \| \|\X_{a,b}\| \\
	&&+ C_p \Big\{\ltn x \rtn_{\tp,[a,b]} \ltn R^{\Phi(c-\cdot)g(y_\cdot)} \rtn_{\tq,[a,b]} + \ltn \X \rtn_{\tq,[a,b]} \ltn [\Phi(c-\cdot)g(y_\cdot)]^\prime \rtn_{\tp,[a,b]}\Big\}\\
	&\leq& C_A e^{-\lambda_A(c-a)} (C_g \|y_a\| + \|g(0)\|) \Big(\ltn x\rtn_{\tp,[a,b]} + C_g \ltn \X\rtn_{\tq,[a,b]}\Big)\\
	&& + C_p C_A e^{-\lambda_A(c-b)} \ltn \X \rtn_{\tq,[a,b]} \Big[C_g^2 \ltn y \rtn_{\tp,[a,b]} + C_g \|A\|(b-a)(C_g \|y\|_{\tp,[a,b]}+ \|g(0)\|) \Big]\\
	&& + C_p C_A e^{-\lambda_A(c-b)} \ltn x \rtn_{\tp,[a,b]} \Big[C_g \ltn R^y \rtn_{\tq,[a,b]} +  \|A\|(b-a)(C_g \|y\|_{\tp,[a,b]}+ \|g(0)\|)  \Big]\\
	&\leq& C_A [1+ C_p\|A\|(b-a)]e^{-\lambda_A(c-b)} (\|y_a\| + \frac{\|g(0)\|}{C_g}) \Big(C_g\ltn x\rtn_{\tp,[a,b]} + C_g^2 \ltn \X\rtn_{\tq,[a,b]} \Big)\\
	&& + C_p C_A e^{-\lambda_A(c-b)} \Big(C_g^2\ltn \X \rtn_{\tq,[a,b]} \vee C_g  \ltn x \rtn_{\tp,[a,b]} \Big) \ltn y,R \rtn_{\tp,[a,b]}
\end{eqnarray*}
which derives \eqref{int2}.

\end{proof}	

\begin{proposition}\label{RDEgdiff}
	 The following estimate holds: for any $0\leq a<b\leq c$
	\begin{equation}\label{RDEvardiff1}
	\Big\|\int_a^{b}\Phi (c-s)[g( y_s+z_s)-g(y_s)] d  x_s\Big\|
	\leq e^{-\lambda_A(c-b)}  \kappa(\bx,[a,b])\Lambda(\bx,[a,b]) \Big(\|z_a\| + \ltn z,R \rtn_{\tp,[a,b]}\Big) ,
	\end{equation}
	where $\kappa$ and $\Lambda$ are given by \eqref{kappa2} and \eqref{Lambda} respectively.
\end{proposition}
\begin{proof}
		By \cite[Proposition 3.3]{duc20}, $g(\ty) - g(y)$ is controlled by $x$ with 	
		\begin{eqnarray*}
		[g(\ty) - g(y)]^\prime_s &=& Dg(\ty_s)g(\ty_s) - Dg(y_s)g(y_s); \\
		\ltn [g(\ty) - g(y)]^\prime\rtn_{\tp,[s,t]} 
		&\leq& 2C_g^2 \Big(\ltn z \rtn_{\tp,[s,t]} + \|z\|_{\infty,[s,t]} \ltn y \rtn_{\tp,[s,t]}\Big);\\
		\ltn R^{g(\ty)-g(y)}\rtn_{\tq,[s,t]} &\leq& C_g \ltn R^{z} \rtn_{\tq,[s,t]} + C_g \|z\|_{\infty,[s,t]} \ltn R^{y} \rtn_{\tq,[s,t]} \\
		&&+ \frac{1}{2}C_g^2 \ltn x \rtn_{\tp,[s,t]} \Big[\ltn z \rtn_{\tp,[s,t]} + \|z\|_{\infty} \Big(\ltn \ty \rtn_{\tp,[s,t]}+\ltn y \rtn_{\tp,[s,t]}\Big)\Big].
	\end{eqnarray*}
Similarly, it is easy to show that $\Phi(c-\cdot)\Big(g(y_\cdot)-g(y)\Big)$ is controlled by $x$ with 
\begin{eqnarray*}
\Big[\Phi(c-\cdot)\Big(g(y_\cdot)-g(y)\Big)\Big]^\prime_s &=& \Phi(c-s)\Big[g(y)-g(y)\Big]^\prime_s;\\
\ltn \Big[\Phi(c-\cdot)\Big(g(y_\cdot)-g(y)\Big)\Big]^\prime \rtn_{\tp,[s,t]} &\leq& \|\Phi(c-\cdot)\|_\infty 	\ltn [g(\ty) - g(y)]^\prime\rtn_{\tp,[s,t]} \\
&& + \ltn \Phi(c-\cdot) \rtn_{\tp,[s,t]} \| [g(\ty) - g(y)]^\prime \|_{\infty,[s,t]}; \\
\ltn \|R^{\Phi(c-\cdot)[g(\ty_\cdot)-g(y)]}\rtn_{\tq,[s,t]} &\leq& C_g \ltn \Phi(c-\cdot) \rtn_{\tq,[s,t]} \|z\|_{\infty,[s,t]} \\
&&+ \| \Phi(c-\cdot) \|_{\infty,[s,t] } 	\ltn R^{g(\ty)-g(y)}\rtn_{\tq,[s,t]}.
\end{eqnarray*} 
The rest is then similar to Proposition \ref{YDEg} and will be omitted.

\end{proof}

The following lemmas are followed directly from \cite[Lemma 3.3]{duchong19} using Proposition \ref{YDEg} and Proposition \ref{RDEgdiff} (the proof is omitted due to similarity).
\begin{lemma}\label{ytest}
	Assume that $y_t$ satisfies \eqref{variation}. Then for any $n\geq 0$, 
	\begin{eqnarray}\label{yt}
		\|y_t\|e^{(\lambda_A-L_f) t}
	&\leq&  C_A \|y_0\| + \frac{C_A}{\lambda_A - L_f} \|f(0)\| \Big(e^{(\lambda_A - L_f)t}-1\Big) \\
	&&+e^{\lambda_A}\sum_{k=0}^{n} e^{(\lambda_A - L_f)k}\kappa(\bx,\Delta_k)  \Big[ \|y_k\| + \frac{\|g(0)\|}{C_g}+ \ltn y,R\rtn_{p{\rm -var},\Delta_k}\Big], \quad \forall t \in \Delta_n\notag
	\end{eqnarray}
where $\Delta_k:=[k,k+1]$, $L_f := C_A C_f$. 
\end{lemma}

\begin{lemma}\label{ztest}
	Assume that $z_t$ satisfies \eqref{variationDiff}. Then for any $n\geq 1$, 
	\begin{equation}\label{zt}
	\|z_n\|e^{(\lambda_A-L_f) n}
	\leq  C_A \|z_0\| + e^{\lambda_A}\sum_{k=0}^{n-1} e^{(\lambda_A - L_f)k}\kappa(\bx,\Delta_k) \Lambda(\bx,\Delta_k) \Big[ \|z_k\| + \ltn z,R\rtn_{p{\rm -var},\Delta_k}\Big].
	\end{equation}

\end{lemma}
We remind the reader of the so-called discrete Gronwall lemma \cite[Appendix 4.2]{duchong19}.
\begin{lemma}[Discrete Gronwall Lemma]\label{gronwall}
	Let $a$ be a non-negative constant and $u_n, \alpha_n,\beta_n$ be  non-negative sequences satisfying 
	$$u_n\leq a + \sum_{k=0}^{n-1} \alpha_ku_k + \sum_{k=0}^{n-1} \beta_k,\;\; \forall n \geq 1$$
	then 
	\begin{equation}\label{estu}
	u_n\leq \max\{a,u_0\}\prod_{k=0}^{n-1} (1+\alpha_k) + \sum_{k=0}^{n-1}\beta_k\prod_{j=k+1}^{n-1}(1+\alpha_j) 
	\end{equation}
	for all $n\geq 1$.
\end{lemma}

We are now able to formulate the first main result of the paper.

\begin{theorem}\label{attractor}
	Assume that $A$ has all eigenvalues of negative real parts with $\lambda_A$ satisfying \eqref{estphi1} and \eqref{estphi2}, and $f$ is globally Lipschitz continuous such that $\lambda_A > C_fC_A$. Assume further that the driving path $x$ satisfies \eqref{gamma}. Then under the condition
	\begin{equation}\label{criterion}
	\lambda_A - C_A C_f> C_A (1+\|A\|)e^{\lambda_A+4(\|A\|+C_f)} \Big\{\Big[4C_pC_g\Gamma(p)\Big]^p + \Big[4C_pC_g\Gamma(p)\Big]\Big\},
	\end{equation}
	where $\Gamma(p)$ is given by \eqref{gamma}, the random dynamical system $\varphi$ possesses a pullback attractor $\mathcal{A}(\bx)$. 
	
\end{theorem}
\begin{proof}
The proof follows \cite[Theorem 3.4]{duchong19} step by step, so we only sketch here some important details. First, by applying Lemma \ref{ytest} and using the estimates in \eqref{estx2} and \eqref{estxlin}
\begin{eqnarray*}
\|y_k\|+\ltn  y,R\rtn_{p{\rm -var},\Delta_k}\leq \|y_k\|e^{4L}N_k^{\frac{p-1}{p}}(\bx)e^{\alpha N_k(\bx)} + \Big( \frac{\|f(0)\|}{L}+  \frac{1}{C_p} \Big)e^{4L}N_k^{\frac{2p-1}{p}}(\bx)e^{\alpha N_k(\bx)} ,
\end{eqnarray*}
where $N_k(x) := N_{\Delta_k}(x)$, we derive from \eqref{yt} that
\allowdisplaybreaks
\begin{eqnarray}\label{xn}
\| y_n\|e^{\lambda n}&\leq &C_A \|y_0\| +\sum_{k=0}^{n-1} \kappa(\bx,\Delta_k) e^{\lambda_A +4L}N_k^{\frac{p-1}{p}}(\bx)e^{\alpha N_k(\bx)}  e^{\lambda k}  \|y_k\|  \\
&&+e^{\lambda_A }\sum_{k=0}^{n-1} \kappa(\bx,\Delta_k) e^{\alpha N_k(\bx)} e^{\lambda k} \Big\{\Big( \frac{\|f(0)\|}{L}+  \frac{1}{C_p} \Big)e^{4L}N_k^{\frac{2p-1}{p}}(\bx) + \frac{\|g(0)\|}{C_g} + \|f(0)\|\Big\}.\notag
\end{eqnarray}
Assign $a:= C_A\|y_0\|$, $u_k:= \| y_k\| e^{\lambda k}, \; k\geq 0$ and 
	\begin{eqnarray}
G(\bx,[a,b]) &:=& e^{\lambda_A+4L(b-a)}\kappa(\bx,[a,b]) N_{[a,b]}^{\frac{p-1}{p}}(\bx)e^{\alpha N_{[a,b]}(\bx)} ,\label{G}\\
H(\bx,[a,b]) &:=& M_0  e^{\lambda_A+4L(b-a)} \kappa(\bx,[a,b])e^{\alpha N_{[a,b]}(\bx)}  \Big[N_{[a,b]}^{\frac{2p-1}{p}}(\bx) + 1\Big], \label{H}
\end{eqnarray}
then \eqref{xn} has the form
\begin{eqnarray}
u_n
\leq a + \sum_{k=0}^{n-1}  G(x,\Delta_k)  u_k + \sum_{k=0}^{n-1} e^{\lambda k} H (x,\Delta_k).
\end{eqnarray}
We are now in the position to apply Lemma \ref{gronwall}, so that
\begin{equation}\label{y}
	\|y_n(\bx,y_0)\|\leq C_A\|y_0\|e^{-\lambda n}  \prod_{k=0}^{n-1} \Big[1+ G(\theta_k x,[0,1])\Big]+\sum_{k=0}^{n-1} e^{-\lambda (n-k)} H(\theta_k x, [0,1]) \prod_{j=k+1}^{n-1} \Big[1+ G(\theta_jx,[0,1])\Big].
\end{equation}
Now for any $t\in [n,n+1]$ by assigning $x$ with $\theta_{-t}x$ and using \eqref{y} we obtain
\begin{equation}\label{yy0}
\|y_t( \theta_{-t}x,y_0(\theta_{-t}x)) \|
\leq C_Ae^{4L}\|y_0(\theta_{-t}x)\|e^{-\lambda n}  \sup_{\epsilon\in[0,1]}  \prod_{k=1}^{n} \Big[1+ G(\theta_{-k} x,[-\epsilon,1-\epsilon])\Big]+ M_0 e^{4L}b(\bx),
\end{equation}
where
\begin{equation}\label{bx}
b(\bx):=\sup_{\epsilon\in[0,1]} \sum_{k=1}^{\infty}e^{-\lambda k}  H(\theta_{-k}x,[-\epsilon,1-\epsilon])\prod_{j=1}^{k-1} \Big[1+G( \theta_{-j}x,[-\epsilon,1-\epsilon])\Big].	
\end{equation}
A direct computation using \eqref{Nest} shows that
\begin{eqnarray*}
G(\bx,[-\epsilon,1-\epsilon]) 
&\leq& \frac{1}{2}C_A(1+|A|)e^{\lambda_A+4L}\Big[ 4C_pC_g \ltn \bx \rtn_{\tp,[-\epsilon,1-\epsilon]}+ [4C_pC_g]^p \ltn \bx \rtn^p_{p{\rm -var},[-\epsilon,1-\epsilon]}\Big]\times \\
&& \times \Big(1 \vee 4C_pC_g \ltn \bx \rtn_{\tp,[-\epsilon,1-\epsilon]}\Big)e^{\alpha [4C_pC_g]^p \ltn \bx \rtn^p_{p{\rm -var},[-\epsilon,1-\epsilon]}}.
\end{eqnarray*}
Applying the inequality
\[
\log\Big[1 + \frac{1}{2}M (1\vee a)(a + a^p)e^{\alpha a^p}\Big] \leq M (a +a^p),\quad \forall a\geq 0, M \geq 1,
\]
and the Birkhorff ergodic theorem, we obtain that
\allowdisplaybreaks
\begin{eqnarray*}
	&&\limsup \limits_{n \to \infty}\frac{1}{n}\log  \sup_{\epsilon\in[0,1]} \prod_{k=0}^{n-1} \Big[1+G(\theta_{-k}\bx,[-\epsilon,1-\epsilon])\Big] \\
	&=& \limsup \limits_{n \to \infty}\sup_{\epsilon\in[0,1]}\frac{1}{n}\sum_{k=0}^{n-1}\log  \Big[1+G(\theta_{k}\bx,[-\epsilon,1-\epsilon])\Big]  \\
	&\leq&  \sup_{\epsilon\in[0,1]}C_Ae^{\lambda_A+4L} \limsup \limits_{n \to \infty} \Big\{ \sup_{\epsilon\in[0,1]}\frac{1}{n}\sum_{k=0}^{n-1}  \ltn \theta_{-k}\bx \rtn^p_{p{\rm -var},[-\epsilon,1-\epsilon]} +  \sup_{\epsilon\in[0,1]}\frac{1}{n}\sum_{k=0}^{n-1}  \ltn \theta_{-k}\bx \rtn_{p{\rm -var},[-\epsilon,1-\epsilon]}\Big\}\\
	&\leq& C_Ae^{\lambda_A+4L}  \Big\{\Big[4C_pC_g\Gamma(\bx,p)\Big]^p + \Big[4C_pC_g\Gamma(\bx,p)\Big]\Big\} = \hat{G},
\end{eqnarray*} 
for a.s. all $\bx$. As a result under condition \eqref{criterion}, for $t\in \Delta_n$ with $0<\delta< \frac{1}{2}(\lambda-\hat{G})$ and $n$ large enough
\begin{equation}\label{yfinal}
\|y_t(\theta_{-t}\bx,y_0)\|
\leq C_Ae^{4L}\|y_0(\theta_{-t}x)\|\exp\left\{-\left(\lambda-\hat{G} -\delta\right)n\right\}+M_0e^{4L}b(\bx).
\end{equation}
Similar to \cite[Proposition 3.5]{duchong19}, one can prove that $b(\bx)$ is finite almost surely and tempered. This implies that, starting from any point $y_0(\theta_{-t}\bx) \in D(\theta_{-t}\bx)$ which is tempered due to \eqref{tempered}, there exists $n$ large enough such that for $t\in[n,n+1]$ 
\begin{equation}\label{bhat}
\|y_t(\theta_{-t}\bx,y_0)\|\leq 1+ M_0 e^{4L}b(\bx)
=: \hat{b}(\bx).
\end{equation}
Moreover, the temperedness of $\hat{b}(\bx)$ follows directly from the temperedness of $b(\bx)$. Therefore, there
exists a compact absorbing set $\mathcal{B}(\bx) = \bar{B}(0,\hat{b}(\bx))$ and thus a pullback attractor $\mathcal{A}(\bx)$ for system \eqref{fSDE0} which is given by \eqref{at}. 
\end{proof}

\begin{remark}\label{comparison}
	(i), Assume that $f(0) = g(0) = 0$ so that $y \equiv 0$ is a solution of \eqref{fSDE0}. Then \eqref{criterion} in Theorem \ref{attractor} is the exponential stability criterion for the trivial attractor $\mathcal{A}(\bx) \equiv 0$.
	
	(ii) It is important to note that the term $e^{\lambda_A+4(\|A\|+C_f)}$ in \eqref{criterion} is the unavoidable effect from the discretization scheme.
\end{remark}

In the rest of the paper, we are going to prove the same results on singleton attractors as in \cite{duchong19} for Young equations. First of all, the conclusion is very explicit in case $g$ is linear, as shown below.


\begin{theorem}\label{linear}
	Assume that $g(y) = Cy + g(0)$ is a linear map, with $\|C\| \leq C_g$. Then under the condition
	\begin{equation}\label{criterionlin}
	\lambda_A - C_AC_f> C_A[1+\|A\|]e^{\lambda_A+4(\|A\|+C_f)}  \Big\{\Big[4C_pC_g\Gamma(p)\Big] +\Big[4C_p C_g\Gamma(p)\Big]^p\Big\},
	\end{equation}
the attractor is a random singleton, i.e. $\mathcal{A}(\bx) = \{a(\bx)\}$ a.s. Moreover, $\mathcal{A}$ is both a pullback and forward attractor.
\end{theorem}
\begin{proof}
 The existence of the pullback attractor $\mathcal{A}$ is followed by Theorem \ref{attractor}. Take any two points $a_1,a_2\in \mathcal{A}(x) $. For a given $n \in \N$, assign $x^*:=\theta_{-n}x$  and consider the equation
\begin{equation}\label{equ.star1}
dy_t = [Ay_t + f(y_t)]dt + g(y_t)dx^*_t.
\end{equation}
Due to the invariance of $\mathcal{A}$ under the flow, there exist $b_1,b_2\in \mathcal{A}(x^*)$ such that $a_i=y_n(x^*,b_i)$.
Put $z_t= z_t(x^*):= y_t(x^*,b_1)- y_t(x^*,b_2)$ then $z_n(x^*) =a_1- a_2$
and we have 
\begin{eqnarray}\label{zequ}
dz_t= [Az_t + P(t,z_t)]dt + Q(t,z_t)dx^*_t
\end{eqnarray}
where we write in short $y^1_t = y_t(x^*,b_1)$ and
\begin{eqnarray*}
	P(t,z_t)&=&f(y(t,x^*,b_2))-f(y(t,x^*,b_1))= f(y^1_t+z_t)-f(y^1_t),\\
	Q(t,z_t)&=&g(y(t,x^*,b_2))-g(y(t,x^*,b_1))=g(y^1_t+z_t)-g(y^1_t).
\end{eqnarray*}	  
For $g(y) = Cy$, observe that $Q(t,z_t) = C z_t$ and
\begin{eqnarray*}
\ltn [\Phi(c-\cdot)Cz_\cdot]^\prime \rtn_{\tp,[a,b]}&\leq&C_g^2 \|\Phi(c-\cdot)\|_{\infty,[a,b]} \ltn z\rtn_{\tp,[a,b]} + C_g^2\ltn \Phi(c-\cdot)\rtn_{\tp,[a,b]} \ltn z \rtn_{\infty,[a,b]} \\
\ltn R^{\Phi(c-\cdot)Cz_\cdot} \rtn_{\tq,[a,b]} &\leq& C_g\|\Phi(c-\cdot) \|_{\infty,[a,b]} \ltn R^z\rtn_{\tq,[a,b]} + C_g\ltn \Phi(c-\cdot) \rtn_{\tq,[a,b]} \|z\|_{\infty,[a,b]}. 
\end{eqnarray*}
As a result, the estimate in \eqref{int2} is of the form
\allowdisplaybreaks
\begin{eqnarray}\label{zt2b}
&&\Big\|\int_a^b \Phi(c-s) Q(s,z_s) dx^*_s \Big\| \notag\\
&\leq& \|\Phi(c-a) Cz_a\| \ltn x^* \rtn_{\tp,[a,b]} + \|\Phi(c-a)C^2 z_a\| \ltn \X^* \rtn_{\tq,[a,b]} \notag\\
&&+ C_p \Big\{\ltn x^* \rtn_{\tp,[a,b]} \ltn R^{\Phi(c-\cdot)C} \rtn_{\tq,[a,b]} + \ltn \X^* \rtn_{\tq,[a,b]} \ltn [\Phi(c-\cdot)C]^\prime \rtn_{\tp,[a,b]}\Big\} \notag\\
&\leq& C_p C_A[1+\|A\|(b-a)] e^{-\lambda_A(c-b)}\Big(C_g \ltn \bx^* \rtn_{\tp,[a,b]} +C_g^2 \ltn \bx^* \rtn^2_{\tp,[a,b]}\Big) \times\notag\\
&&\times \Big(\|z_a\|+ \ltn z,R^z\rtn_{\tp,[a,b]}\Big).
\end{eqnarray}
Meanwhile, similar arguments in \cite[Theorem 3.9]{duc20}, with $P(t,0) =0$, show that 
\begin{eqnarray*}
\ltn z,R^z \rtn_{\tp,[a,b]} + \|z_a\|&\leq& \|z_a\| e^{4L(b-a) + \alpha N_{[a,b]}(\bx^*)}N^{\frac{p-1}{p}}_{[a,b]}(\bx^*),\quad \text{with} \\
N_{[a,b]}(\bx^*) &\leq& 1+[4C_pC_g]^p \ltn \bx^* \rtn_{p{\rm -var},[a,b]}^p.
\end{eqnarray*}
As a result, 
\begin{eqnarray*}
e^{\lambda n} \|z_n\|&\leq & 
C_A\|z_0\| +C_p C_A[1+\|A\|]e^{\lambda_A}\sum_{k=0}^{n-1} \Big(C_g \ltn \bx^* \rtn_{\tp,\Delta_k} + C_g^2 \ltn \bx^* \rtn^2_{\tp,\Delta_k}\Big) \times \\
&&\times  e^{4L + \alpha N_{\Delta_k}(\bx^*)}N^{\frac{p-1}{p}}_{\Delta_k}(\bx^*) e^{\lambda k}\|z_k\| \\
&\leq&C_A\|z_0\| + \sum_{k=0}^{n-1}I(\bx^*,\Delta_k)e^{\lambda k} \|  z_k\|,
\end{eqnarray*}
where
\begin{eqnarray*}
I(\bx,[a,b])=C_p C_A[1+\|A\|]e^{\lambda_A}\Big(C_g \ltn \bx \rtn_{\tp,[a,b]} + C_g^2 \ltn \bx \rtn^2_{\tp,[a,b]}\Big)e^{4L + \alpha N_{[a,b]}(\bx)}N^{\frac{p-1}{p}}_{[a,b]}(\bx) 
\end{eqnarray*}
is a function of $\bx$. We now apply the discrete Gronwall lemma \ref{gronwall} to obtain
\begin{eqnarray*}
e^{\lambda n} \|z_n\|&\leq &C_A\|z_0\| \prod_{k=0}^{n-1}\Big[1+I(\theta_{k-n}\bx,[0,1])\Big].
\end{eqnarray*}
Hence, it follows from Birkhorff's ergodic theorem that
\begin{eqnarray*}
	\limsup \limits_{n \to \infty} \frac{1}{n}\log \|z_n\| &\leq &-\lambda + \limsup \limits_{n \to \infty} \frac{1}{n} \sum_{k=1}^n \log \Big[1+I(\theta_{-k}\bx,[0,1])\Big] \leq -\lambda + E \log\Big[1+ I(\bx,[0,1])\Big].
\end{eqnarray*}
Given $C_p$ and $\alpha$, it follows from the estimate of $N_{[0,1]}(\bx)$ and the inequalities 
\begin{eqnarray*}
\log(1+ u e^v) &\leq& v+\log(1+u) , \quad \forall u, v \geq 0, \\
\log\Big[ 1+ \frac{(2C_p+3)v}{4C_p} (1+u^{p-1})(u+u^2) \Big] &\leq& (2-\alpha)v (u+u^p),\quad \forall u \geq 0, v \geq 1,
\end{eqnarray*}
that
\begin{eqnarray*}
&&\log \Big[1+I(\bx,[0,1])\Big]\\
&\leq & \log\Big\{1+C_p C_A[1+\|A\|]e^{\lambda_A} e^{4L + \alpha}N^{\frac{p-1}{p}}_{[a,b]}(\bx)\Big(C_g \ltn \bx \rtn_{\tp,[a,b]} + C_g^2 \ltn \bx \rtn^2_{\tp,[a,b]}\Big) \Big\} \\
&&+ \alpha \Big(4C_pC_g\Big)^p \ltn \bx \rtn^p_{\tp,[0,1]} \\
&\leq& \log\Big\{1+ \frac{2C_p+3}{4C_p}C_A[1+\|A\|]e^{\lambda_A+4L} \Big[1+ (4C_pC_g)^{p-1} \ltn \bx \rtn^{p-1}_{\tp,[0,1]}\Big] \times\\
&&\times \Big(4C_pC_g \ltn \bx \rtn_{\tp,[a,b]} + \Big[4C_pC_g\Big]^2 \ltn \bx \rtn^2_{\tp,[a,b]}\Big)\Big\} + \alpha \Big(4C_p\|C\|\Big)^p \ltn \bx \rtn^p_{\tp,[0,1]} \\
&\leq& C_A(1+\|A\|)e^{\lambda_A+4L}  \Big\{\Big[4C_pC_g\ltn \bx \rtn_{\tp,[0,1]}\Big] +\Big[4C_pC_g\ltn \bx \rtn_{\tp,[0,1]}\Big]^p \Big\}
\end{eqnarray*}
Therefore we finally obtain
\begin{eqnarray*}
	\limsup \limits_{n \to \infty} \frac{1}{n}\log \|z_n\|&\leq & -\lambda  +C_A(1+\|A\|)e^{\lambda_A+4L}  \Big\{\Big[4C_pC_g\Gamma(p)\Big] +\Big[4C_pC_g\Gamma(p)\Big]^p\Big\} <0
\end{eqnarray*}
under the condition \eqref{criterionlin}. This follows that $\lim_{n\to\infty} \|a_1-a_2\|=0$ or $\mathcal{A}$ is an one point set. The arguments in the forward direction of convergence to $\mathcal{A}$ are similar.
\end{proof}
	
The next case of bounded $g \in C^3_b$ is a little more technical, but applies the same idea as in \cite[Theorem 3.11]{duchong19}.	
\begin{theorem}\label{gbounded}
	Assume that $g \in C^3_b$ and $\lambda_A > C_fC_A$. Then there exists a $\delta >0$ small enough such that for any $C_g \leq \delta$ the attractor is a singleton. Moreover, $\mathcal{A}$ is both a pullback and forward attractor.
\end{theorem}
\begin{proof}
	We follow line by line three steps in the proof of \cite[Theorem 3.12]{duchong19}. \\
	
	 {\bf Step 1}. Similar to \cite[Theorem 3.11]{duchong19}, we prove that there exists a time $r>0$, a constant $\eta  \in (0,1)$, and an integrable random variable $\xi_r(x)$ such that
	 \begin{equation}\label{ydissipative}
	 \|y_r(x,y_0)\|^{p} \leq \eta \|y_0\|^{p} + \xi_r(x).
	 \end{equation}
	 First denote by $\mu_t$ the solution of the deterministic system $\dot{\mu} = A\mu + f(\mu)$ which starts at $\mu_0 = y_0$.  Using \eqref{yt} we obtain for any fixed $r>0$ the estimates
	\begin{eqnarray}
	\|\mu_t\| &\leq& C_A\|y_0\| e^{-\lambda t} + C_A\frac{\|f(0)\|}{\lambda},\\
	\|\mu_{s,t}\| &\leq& \int_s^t (\|f(0)\| + L \|\mu_u\|)du \leq D(1+ \|y_0\|) (t-s),\forall 0 \leq s < t \leq r, \label{muest1}
	\end{eqnarray}
	for a generic constant $D$ independent of $\|y_0\|$ from now on. Define $h_t = y_t - \mu_t$, then $h$ satisfies
	\[
	h_{s,t} = \int_s^t \Big(Ah_u + f(h_u+\mu_u) - f(\mu_u)\Big) du + \int_s^t g(h_u + \mu_u) dx_u,
	\]
	which implies that $h$ is also controlled by $x$ with $h^\prime_s = y^\prime_s = g(h_s + z_s)$ and $R^h_{s,t} = R^y_{s,t} - z_{s,t}$. In addition, 
	\begin{eqnarray}\label{hest1}
	\|h_{s,t}\| &\leq& \int_s^t L \|h_u\| du + C_g \|x_{s,t}\| + C^2_g \|\X_{s,t}\| \notag\\
	&& + C_p \Big\{ \ltn \X \rtn_{\tq,[s,t]} \ltn [g(y)]^\prime \rtn_{\tp,[s,t]} + \ltn x \rtn_{\tp,[s,t]} \ltn R^{g(y)}\rtn_{\tq,[s,t]} \Big\}
	\end{eqnarray}
	Now observe that for $\beta = \frac{2}{p} \in (\frac{2}{3},1)$,
	\begin{eqnarray}\label{muest2}
	&& \|g(\mu_u + h_u) - g(\mu_v +h_v)\| \vee \|Dg(\mu_u + h_u) - Dg(\mu_v +h_v)\| \notag\\
	&\leq& C_g \|h_{u,v}\| + 2C_g \|\mu_{u,v}\|^\beta \notag\\
	&\leq& C_g \|h_{u,v}\| + D (1+ \|y_0\|^\beta) (t-s)^\beta, \forall 0 \leq u < v \leq r.
	\end{eqnarray}
	Hence \eqref{muest2} derives 
	\begin{equation}\label{muest3}
	\ltn g(\mu +h)\rtn_{p{\rm-var},[s,t]} \vee	\ltn Dg(\mu +h)\rtn_{p{\rm-var},[s,t]} 
	 \leq C_g \ltn h \rtn_{p{\rm-var},[s,t]} + D (1+ \|y_0\|^\beta)(t-s)^\beta,\quad \forall 0 \leq s < t \leq r. 
	\end{equation}
	Inequality \eqref{muest3} together with $[g(y)]^\prime_s = Dg(y_s) g(y_s)$ leads to
	\[
	\ltn [g(y)]^\prime \rtn_{\tp,[s,t]} \leq 2 C_g^2 \Big(\ltn h \rtn_{\tp,[s,t]} +D(1+ \|y_0\|^\beta) r^\beta\Big),\quad \forall 0\leq s<t \leq r.
	\]
	Furthermore, for all $0 \leq s < t \leq r$,
	\begin{eqnarray*}
	\|R^{g(y)}_{s,t}\| &=&  \|g(y_s + h^\prime_s \otimes x_{s,t} + R^h_{s,t} + \mu_{s,t}) - g(y_s) - Dg(y_s) g(y_s) \otimes x_{s,t}\| \notag\\
	&\leq&  \|g(y_s + g(y_s) \otimes x_{s,t} + R^h_{s,t} + \mu_{s,t}) - g(y_s + h^\prime_s \otimes x_{s,t})\| \notag\\
	&& + \|g(y_s + g(y_s) \otimes x_{s,t})- g(y_s) - Dg(y_s) g(y_s) \otimes x_{s,t}\| \notag\\
	&\leq& C_g \|R^h_{s,t}\| + 2C_g \|\mu_{s,t}\|^\beta + \int_0^1 \|D_g(y_s + \eta g(y_s) \otimes x_{s,t}) - Dg(y_s)\| \|g(y_s)\| \|x_{s,t}\| d \eta \notag\\
	&\leq& C_g \|R^h_{s,t}\| +  D (1+ \|y_0\|^\beta)(t-s)^\beta + \frac{1}{2}C_g^3 \|x_{s,t}\|^2,
	\end{eqnarray*}
	which leads to
	\begin{equation}\label{muest4}
	\ltn R^{g(y)} \rtn_{\tq,[s,t]} \leq C_g \ltn R^h \rtn_{\tq,[s,t]} +   D(1+ \|y_0\|^\beta) r^\beta + \frac{1}{2}C_g^3 \ltn x \rtn_{\tp,[s,t]}
	\end{equation}
	 for all $0 \leq s < t \leq r$.	Replacing \eqref{muest3} and \eqref{muest4} into \eqref{hest1} we obtain
	\begin{eqnarray*}
	&&\ltn h\rtn_{\tp,[s,t]}\\
	&\leq& \int_s^t L \|h_u\| du + C_g \ltn x\rtn_{\tp,[s,t]} + C^2_g \ltn \X \rtn_{\tq,[s,t]} + \frac{1}{2} C_p C_g^3 \ltn x\rtn_{\tp,[s,t]}^3\notag\\
	&& + C_p (2C_g^2\ltn \X \rtn_{\tq,[s,t]} + C_g\ltn x \rtn_{\tp,[s,t]}) \Big\{\ltn h,R^h\rtn_{\tp,[s,t]} +  D (1+ \|y_0\|^\beta)r^\beta \Big\}.
	\end{eqnarray*}
	The same estimate for $R^h$ also holds
	\begin{eqnarray*}
	&& \ltn R^h\rtn_{\tq,[s,t]}\\
	 &\leq& \int_s^t L \|h_u\| du  + C^2_g \ltn \X \rtn_{\tq,[s,t]} + \frac{1}{2} C_p C_g^3 \ltn x\rtn_{\tp,[s,t]}^3\notag\\
	&& + C_p (2C_g^2\ltn \X \rtn_{\tq,[s,t]} + C_g\ltn x \rtn_{\tp,[s,t]}) \Big\{\ltn h,R^h\rtn_{\tp,[s,t]} + D(1+ \|y_0\|^\beta)r^\beta \Big\}.
\end{eqnarray*}
Hence we obtain
\begin{eqnarray*}
&& \ltn h,R^h\rtn_{\tp,[s,t]} \notag\\
&\leq&  \int_s^t 2L \|h_u\| du + C_g \ltn x\rtn_{\tp,[s,t]} + 2 C^2_g \ltn \X \rtn_{\tq,[s,t]} +  C_p C_g^3 \ltn x\rtn_{\tp,[s,t]}^3\notag\\
&& + 2 C_p (2C_g^2\ltn \X \rtn_{\tq,[s,t]} + C_g\ltn x \rtn_{\tp,[s,t]}) \Big\{\ltn h,R^h\rtn_{\tp,[s,t]} +  D(1+ \|y_0\|^\beta)r^\beta \Big\}.
\end{eqnarray*}
Using similar arguments to the proof of Theorem \ref{RDE}, we can prove that
	\[
	\|h\|_{\infty,[s,t]} \leq e^{4L(t-s)} \Big[ \|h_s\| + \Big(\frac{1}{C_p} +D(1+ \|y_0\|^\beta)r^\beta\Big) N_{[s,t]}(\bx)\Big],\forall 0 \leq s < t \leq r.
	\]
	Since $h_0 = y_0 - \mu_0 = 0$, it follows that
	\begin{eqnarray}\label{hest2}
	\|h_r\| \leq \|h\|_{\infty,[0,r]} \leq e^{4Lr}N_{[0,r]}(\bx)  \Big(\frac{1}{C_p} + D(1+ \|y_0\|^\beta)r^\beta\Big) \leq \xi (\bx) (1+  \|y_0\|^\beta)
	\end{eqnarray}
	where $\xi$ is a linear form of $N_{[0,r]}(\bx)$.  Similar to \cite[Theorem 3.11]{duchong19}, we apply, for $\epsilon>0$ small enough, the convex inequality and Young inequality to conclude that 
	\allowdisplaybreaks
	\begin{eqnarray}\label{hest3}
	\|y_r\|^{2p} \leq (\|h_r\| + \|\mu_r\|)^{2p} \leq (1+\epsilon)^{2(2p-1)} \Big[(C_Ae^{-\lambda r})^{2p}  + \epsilon \beta \Big] \|y_0\|^{2p} + \xi_{r}(\bx),
	\end{eqnarray}
	for some function $\xi_{r}(\bx)$ of $\bx$, which is dependent on $\epsilon$ and integrable due to the assumption. By choosing $r >0$ large enough and $\epsilon \in (0,1)$ small enough such that  
	\[
	C_Ae^{-\lambda r} < 1 \quad \text{and} \quad \eta : = (1+\epsilon)^{2(2p-1)} \Big[(C_Ae^{-\lambda r})^{2p}  + \epsilon \beta \Big]<1,
	\]
	we obtain \eqref{ydissipative}.
	 \\
	
		{\bf Step 2.} Next, for simplicity we only estimate $y$ at the discrete times $nr$ for $n \in \N$, the estimate for $t \in [nr,(n+1)r]$ is similar to \eqref{yy0}. From \eqref{ydissipative}, it is easy to prove by induction that
	\[
	\|y_{nr}(x,y_0)\|^{2p} \leq \eta^n \|y_0\|^{2p} + \sum_{i = 0}^{n-1} \eta^i \xi_r(\theta_{(n-i)r}x), \quad \forall n \geq 1;
	\]
	thus for $n$ large enough 
	\[
	\|y_{nr}(\theta_{-nr}x,y_0)\|^{2p} \leq \eta^n \|y_0\|^{2p} + \sum_{i = 0}^n \eta^i \xi_r(\theta_{-ir}x) \leq 1 + \sum_{i =0}^\infty \eta^i \xi_r(\theta_{-ir}x) =: R_r(x).
	\]
	In this case we could choose $\hat{b}(x)$ in \eqref{bhat} to be $\hat{b}(x)=R_r(x)^{\frac{1}{2p}}$ so that there exists a pullback absorbing set $\cB(x) = B(0,\hat{b}(x))$ containing our random attractor $\mathcal{A}(x)$. Moreover, due to the integrability of $\xi_r(x)$, $R_r(x)$ is also integrable with $\E R_r = 1 + \frac{\E \xi_r}{1-\eta}$.

	{\bf Step 3.} Now back to the arguments in the proof of Theorem \ref{linear} and using Lemma \ref{ztest} and \eqref{roughest4}, we obtain
	\begin{eqnarray}\label{zt2}
	\|z_n\|e^{\lambda n}
	&\leq&  C_A \|z_0\| + e^{\lambda_A+4L}\sum_{k=0}^{n-1} \kappa(\bx^*,\Delta_k) \Lambda(\bx^*,\Delta_k) \Big(1+ \Big[8C_p C_g \Lambda(\bx^*,\Delta_k)\Big]^{p-1} \ltn \bx^* \rtn_{p{\rm -var},\Delta_k}^{p-1} \Big)e^{\lambda k} \|z_k\| \notag \\
	&\leq& C_A\|z_0\| + \sum_{k=0}^{n-1}I(\bx^*,\Delta_k)e^{\lambda k} \|  z_k\|.
	\end{eqnarray}
	We now apply the discrete Gronwall lemma \ref{gronwall} to obtain
	\begin{eqnarray*}
		e^{\lambda n} \|z_n\|&\leq &C_A\|z_0\| \prod_{k=0}^{n-1}\Big[1+I(\theta_{k-n}\bx,[0,1])\Big].
	\end{eqnarray*}
From now on, the proof applies the same arguments as in the proof of Theorem \ref{linear}, so we only focus on proving that $\E \log\Big[1+ I(\bx,[0,1])\Big] <\infty$. Assign $\bar{\Lambda}(\bx,[a,b]) :=8C_p C_g\Lambda(\bx,[a,b]) \ltn \bx \rtn_{p{\rm -var},[a,b]}$. Using \eqref{roughest4} with $b_1(\bx),b_2(\bx)$, H\"older inequality and Cauchy inequality, it follows, with a generic polynomial $M(\bx,[0,1])$ of $\ltn \bx \rtn_{p{\rm -var},[0,1]}$, that
\begin{eqnarray}\label{eqI}
I(\bx,[0,1]) &\leq& \frac{1}{4}C_pC_A(1+|A|) \Big(1 \vee C_g \ltn \bx \rtn_{\tp,[0,1]} \Big) \Big(\bar{\Lambda}(\bx,[0,1])  + \bar{\Lambda}^p(\bx,[0,1]) \Big) \notag\\
&\leq& \Big(1 + \bar{\Lambda}^p(\bx,[0,1]) \Big)M(\bx,[0,1]) \notag\\
&\leq& \Big(1 + \Lambda^p(\bx,[0,1]) \Big)M(\bx,[0,1]) \notag\\
&\leq& \Big[\Big(\|b_1(\bx)\| \vee \|b_2(\bx)\|\Big)^p(\bx) + 1\Big] M(\bx,[0,1]) \notag\\
&\leq& \Big(\hat{b}^p(\bx) + 1\Big) M(\bx,[0,1]) \notag\\
&\leq& \hat{b}^{2p}(\bx) + M(\bx,[0,1]).
\end{eqnarray}
The first term in \eqref{eqI} is integrable due to {\bf Step 2}, while the second term is integrable from the assumption. Therefore $I(\bx,[0,1])$ is integrable and so is $\log\Big[1+ I(\bx,[0,1])\Big]$, which is enough to prove that the pullback attractor is a singleton and is also a forward attractor.
\end{proof}	

\begin{corollary}\label{continuityattractor}
	Denote by $\mu^*$ the unique equilibrium of the deterministic system $\dot{\mu} = A\mu + f(\mu)$. Assume that $\|g\|_\infty \leq C_g $ and the assumptions of Theorem \ref{gbounded} hold so that there exists a singleton attractor $\mathcal{A}(x) = \{a(x)\}$ for $C_g < \delta$ small enough. Then 
	\begin{equation}\label{attractorcont}
	\lim \limits_{C_g \to 0} \|a(x) - \mu^*\| =0 \quad \text{a.s., and}\quad \lim \limits_{C_g \to 0} \E \|a(x) - \mu^*\|^{2p} =0.   
	\end{equation}
\end{corollary}

\begin{proof}
	The proof follows similar arguments in \cite[Corollary 3.12]{duchong19} with similar estimates to {\bf Step 1} of Theorem \ref{gbounded}, so it will be omitted here.
\end{proof}	

\section*{Acknowledgments}
The author would like to thank Vietnam Institute for Advanced Studies in Mathematics (VIASM) for finanical support during three month research stay at the institute in 2019.



\begin{thebibliography}{1}
	\bibitem{amann}
	H. Amann.
	\newblock{\em Ordinary Differential Equations: An Introduction to Nonlinear Analysis.}
	\newblock{Walter de Gruyter, Berlin . New York}, (1990).
	%
	%
	\bibitem{arnold}
	L. Arnold.
	\newblock{\em Random Dynamical Systems.}
	\newblock{Springer, Berlin Heidelberg New York}, (1998).
	%
	\bibitem{BRSch17}
	I. Bailleul, S. Riedel, M. Scheutzow.
	\newblock{Random dynamical systems, rough paths and rough flows.}
	\newblock{\em J. Differential Equations}, Vol. {\bf 262}, (2017), 5792--5823.
	%
	\bibitem{cassetal}
	T. Cass, C. Litterer, T. Lyons.
	\newblock{Integrability and tail estimates for Gaussian rough differential equations.}
	\newblock{\em Annals of Probability}, Vol. {\bf 14}, No. 4, (2013), 3026--3050. 
	%
	%
	%
	\bibitem{crauelkloeden}
	H.~Crauel, P.~Kloeden,
	\newblock{Nonautonomous and random attractors.}
	\newblock{\em Jahresber Dtsch. Math-Ver.} {\bf 117} (2015), 173--206.
	%
	\bibitem{duc20}
	L. H. Duc.
	\newblock{Controlled differential equations as rough integrals.}
	\newblock{\em Preprint arXiv:2007.06295. }
	%
	\bibitem{duchong19}
	L. H. Duc, P. T. Hong.
	\newblock{Asymptotic stability of controlled differential equations. Part I: Young integrals.}
	\newblock{\em Preprint https://arxiv.org/abs/1905.04945.} 
	%
	\bibitem{ducGANSch18}
	L. H. Duc, M. J. Garrido-Atienza, A. Neuenkirch, B. Schmalfu\ss.
	\newblock{Exponential stability of stochastic evolution equations driven by small fractional Brownian motion with Hurst parameter in $(\frac{1}{2},1)$}.
	\newblock{\em J. Differential Equations}, 264 (2018), 1119-1145.
	%
	%
	\bibitem{frizhairer}
		P. Friz, M. Hairer.
		\newblock{\em A course on rough path with an introduction to regularity structure.}
		\newblock{Universitext}, Vol. {\bf XIV}, Springer, Berlin, 2014.
	%
	\bibitem{floris}
	C. Floris.
	\newblock{Stochastic stability of the inverted pendulum subjected to delta- correlated base excitation.}
	\newblock{\em Advances in Engineering Software}, {\bf 120}, (2018), 4--13.	
	%
	\bibitem{friz}
	P. Friz, N. Victoir.
	\newblock {\em Multidimensional stochastic processes as rough paths: theory and applications.}
	\newblock {Cambridge Studies in Advanced Mathematics, 120. Cambridge Unversity Press, Cambridge}, (2010).
	%
	%
	%
	%
	%
	%
	\bibitem{gubinelli}
	M. Gubinelli.
	\newblock {Controlling rough paths.}
	\newblock  {\em J. Funtional Analysis}, {\bf 216} (1), (2004), 86--140.
	%
	\bibitem{hairer03}
	M. Hairer.
	\newblock{Ergodicity  of  stochastic  differential  equations  driven  by  fractional  Brownian  motion.}
	\newblock{\em The Annals of Probability}, Vol. {\bf 33}, (2005), 703--758.
	%
	\bibitem{hairer07}
	M. Hairer, A. Ohashi.
	\newblock{Ergodic  theory  for  sdes  with  extrinsic  memory.}
	\newblock{\em The Annals of Probability}, Vol. {\bf 35}, (2007), 1950--1977.
	%
	\bibitem{hairer11}
	M. Hairer, N. Pillai.
	\newblock{Ergodicity  of  hypoelliptic  sdes  driven  by  fractional  Brownian  motion.}
	\newblock{\em Ann. Inst. Henri Poincar\'e}, Vol. {\bf 47}, (2011), 601--628
	%
	\bibitem{hairer13}
	M. Hairer, N. Pillai.
	\newblock{Regularity of laws and ergodicity of hypoelliptic stochastic differential equations driven by rough paths.}
	\newblock{\em The Annals of Probability}, Vol. {\bf 41}, (2013), 2544--2598.
	%
	\bibitem{khasminskii}  
	R. Khasminskii.
	\newblock{\em Stochastic stability of differential equations.}
	\newblock{Springer, Vol. 66}, 2011.
	%
	%
	\bibitem{ImkSchm01}
	P. Imkeller, B. Schmalfuss.
	\newblock{The conjugacy of stochastic and random differential equations and the
		existence of global attractors.}
	\newblock {\em J. Dyn. Diff. Equat.} {\bf 13}, No. 2, (2001), 215--249.
	%
	\bibitem{KelSchm98}
	H. Keller, B. Schmalfuss.
	\newblock{Attractors for stochastic differential equations with nontrivial
		noise.}
	\newblock {\em Bul. Acad. \c Stiin\c te Repub. Mold. Mat.} {\bf 26}, No. 1, (1998), 43--54.
	%
	\bibitem{lyons94}
	T. Lyons.
	\newblock{Differential equations driven by rough signals. I. An extension of an inequality of L.C. Young.}
	\newblock{\em Math. Res. Lett.}, {\bf 1}, No. 4, (1994), 451--464.
	%
	\bibitem{lyons98}
	T. Lyons.
	\newblock{Differential equations driven by rough signals.}
	\newblock{\em Rev. Mat. Iberoam.}, Vol. {\bf 14} (2), (1998), 215--310.
	%
	\bibitem{lyonsetal07}
	T. Lyons, M. Caruana, Th. L\'evy.
	\newblock{\em Differential equations driven by rough paths.}
	\newblock{Lecture Notes in Mathematics}, Vol. {\bf 1908}, Springer, Berlin 2007.
	%
	%
	\bibitem{mandelbrot}
	B. Mandelbrot, J. van Ness.
	\newblock{Fractional Brownian motion, fractional noises and applications.}
	\newblock{\em SIAM Review}, {\bf 4}, No. 10, (1968), 422--437.
	%
	%
	%
	%
	%
	\bibitem{riedelScheutzow}
	S. Riedel, M. Scheutzow.
	\newblock{\em Rough differential equations with unbounded drift terms.}
	\newblock{J. Differential Equations}, Vol. {\bf 262}, (2017), 283--312.
	%
	\bibitem{Sus78}
	H. J. Sussmann.
	\newblock {On the gap between deterministic and stochastic ordinary differential
		equations.}
	\newblock {\em The Annals of Probability.} {\bf 6}, No. 1, (1978), 19--41.
	%
	\bibitem{young}
	L.C. Young.
	\newblock{An integration of H{\"o}lder type, connected with Stieltjes integration.}
	\newblock{\em Acta Math.} {\bf 67}, (1936), 251--282.
	%
	%
\end{thebibliography}
\end{document}